%% file: main.tex
\documentclass[a4paper]{article}

\usepackage{filecontents}

\usepackage[sorting=nyt,backend=bibtex,style=numeric,giveninits=true,maxbibnames=9,maxcitenames=2,backref=false]{biblatex} 

\DefineBibliographyStrings{english}{%
  backrefpage = {$\uparrow$},
  backrefpages = {$\uparrow$},
}

\bibliography{gerber}

\usepackage[english]{babel}
\usepackage{inputenc}
\usepackage[T1]{fontenc}
\usepackage{csquotes}   

\usepackage[top=2cm,bottom=2.2cm,inner=2cm,outer=2cm]{geometry}

\usepackage[usenames,dvipsnames,svgnames,table]{xcolor}	
\usepackage{soul}							

\usepackage{paralist}							   
\usepackage[shortlabels]{enumitem}        

\usepackage{verbatim}                           
\usepackage{spverbatim} 
\usepackage{listings}                              

\usepackage{graphicx}   
\usepackage[tight]{subfigure}                 
\usepackage{rotating}								
\usepackage[footnotesize]{caption}					
\usepackage{tabulary,tabularx,tabu}					
\usepackage[table]{xcolor}
\usepackage{adjustbox}			

\definecolor{phase}{rgb}{0.57, 0.64, 0.69}
\usepackage{tikz}	
\usetikzlibrary{decorations.pathmorphing, patterns,shapes,decorations.pathreplacing,arrows,automata}		

\usepackage{mathtools}  
\usepackage{amsfonts}   
\usepackage{amssymb}    
\usepackage{dsfont,mathrsfs}    
\usepackage{bm}		
\usepackage{slantsc}	
\usepackage[makeroom]{cancel}		
\usepackage{arydshln} 
\usepackage{latexsym,bbm}

\usepackage{titletoc}		
\usepackage[nottoc,notbib]{tocbibind}	
\usepackage{secdot}      
\usepackage{appendix}	

\usepackage{amsthm}     
\usepackage{xspace}     
\usepackage{pdflscape}  
\allowdisplaybreaks     
\usepackage[shortcuts]{extdash}    
\usepackage{ifpdf}      
\usepackage{ifthen}		
\usepackage{microtype}  
\usepackage{pdfpages}	

\usepackage{setspace}				
\usepackage[citecolor=PineGreen,colorlinks=true,linkcolor=RedViolet]{hyperref}

\usepackage[nottoc]{tocbibind}

\usepackage[nameinlink]{cleveref}		
\newcommand{\crefrangeconjunction}{--}

\newtheorem{dummy}{Dummy}[section]              
\newtheorem{proposition}[dummy]{Proposition}
\Crefname{proposition}{Proposition}{Propositions}
\newtheorem{lemma}[dummy]{Lemma}
\Crefname{lemma}{Lemma}{Lemmas}
\newtheorem{theorem}[dummy]{Theorem}
\Crefname{theorem}{Theorem}{Theorems}
\newtheorem{corollary}[dummy]{Corollary}
\theoremstyle{definition}
\newtheorem{definition}[dummy]{Definition}
\newtheorem{remark}[dummy]{Remark}
\newtheorem{approximation}[dummy]{Approximation}

\include{macros}

\dottedcontents{section}[3em]{}{2em}{1pc}
\dottedcontents{subsection}[5em]{}{2.4em}{1pc}

\newcommand{\footnoteremember}[2]
{
   \newcounter{#1}\footnote{#2}\setcounter{#1}{\value{footnote}}
}

\title{Phase-type approximations perturbed by a heavy-tailed component for the Gerber-Shiu function of risk processes with two-sided jumps}

\author{
    Zbigniew Palmowski\footnoteremember{WR}{Faculty of Pure and Applied Mathematics, Wroc\l aw University of Science and Technology, ul.\ Wyb.\ Wyspia\'nskiego 27, 50-370 Wroc\l aw, Poland.}\\
    \small \texttt{zbigniew.palmowski@gmail.com}\\
    \and
    Eleni Vatamidou\footnoteremember{UNIL}{The Faculty of Business and Economics, University of Lausanne, Quartier UNIL-Chamberonne B\^{a}timent Extranef, 1015 Lausanne, Switzerland.}\\
    \small \texttt{eleni.vatamidou@unil.ch}\\
}

\date{}

\begin{document}

\maketitle

\begin{abstract}
We consider in this paper a risk reserve process where the claims and gains arrive according to two independent Poisson processes. While the gain sizes are phase\-/type distributed, we assume instead that the claim sizes are phase\-/type perturbed by a heavy\-/tailed component; that is, the claim size distribution is formally chosen to be phase\-/type with large probability $1-\epsilon$ and heavy\-/tailed with small probability $\epsilon$. We analyze the seminal Gerber\-/Shiu function coding the joint distribution of the time to ruin, the surplus immediately before ruin, and the deficit at ruin. We derive its value as an expansion with respect to powers of $\epsilon$ with known coefficients and we construct approximations from the first two terms of the aforementioned series. The main idea is based on the so\-/called fluid embedding that allows to put the considered risk process into the framework of spectrally negative Markov\-/additive processes and use its fluctuation theory developed in \cite{ivanovs2012occupation}.
\end{abstract}


\section{Introduction}\label{Section: Introduction}

\citeauthor{gerber1998on} \cite{gerber1998on} investigated the classical insurance risk process
\begin{equation}\label{Eq.Classical risk model}
  \levy = u+ ct - \sum_{k=1}^{\fun[t]{\npp}}\gnc[k],
\end{equation}
where $u\geq 0 $ is the initial surplus, $c$ is the premium rate, $\fun[t]{\npp}$, $t \geq 0$, is a Poisson arrival process, and $\{\gnc[k]\}_{k=1}^\infty$ is a sequence of i.i.d.\ claim sizes independent of $\fun[t]{\npp}$. For this process, they analyzed the function
\begin{equation}\label{Eq.Gerber-Shiu definition}
\gerberq := \emeasure \left[ e^{-q \ruintime}\penaltyf{\deficitruin}{\surpluspriorruin} \indfun[\{\ruintime <\infty\}]\right],
\end{equation}
  where
\begin{equation}\label{Eq.Ruin time}
\ruintime := \inf\{t\geq 0 : \levy <0\}
\end{equation}
 is the ruin time,
 \deficitruin is the {\it deficit at ruin}, \surpluspriorruin is the {\it surplus prior to ruin}, and $\emeasure \left[ \cdot \right] = \e \left[ \cdot | \levy[0] = u \right]$.
 The function  $\penaltyfunction : (0,\infty)^2 \rightarrow [0,\infty)$ is any bounded, measurable penalty function, and $q \geq 0$ is the discount factor. Here the indicator function \indfun[\{\ruintime <\infty\}] emphasizes that the penalty is exercised only when ruin occurs.
This function has found much attention since then and was given the name Gerber\-/Shiu (GS)  or discounted penalty function.

The Gerber\-/Shiu function has been widely used as an important risk measurement tool in the theory of bankruptcy, and at the same time, its existence triggered the research for many problems associated with bankruptcy. There are two main approaches allowing to identify the Gerber\-/Shiu function. The first one is based on solving appropriate (defective) renewal\-/type integro\-/differential equations, which is a consequence of the Markov property of the risk process and the fact that the GS function is harmonic with respect to its infinitesimal generator (see the discussion in \cite{lu2007expected}). The second approach is based on probabilistic methods, mainly related to martingale theory or the so\-/called scale functions in the context of L\'evy\-/type risk processes (see \cite{kyprianou-GSRT} for details).
In this paper, we choose the second method for analyzing the GS function.

Because of the importance of the GS function for risk theory, many authors have contributed to its analysis where the underlying risk reserve process has been generalized in several directions. The classical risk process (perturbed by Brownian motion) has been considered in \cite{chiu2003time, gerber1998on} (see also references therein). Later, Markov modulation (or regime switching) was added to the model; see \cite{asmussen-RP} for a beautiful overview on this topic. The more general so\-/called spectrally negative Markov\-/additive risk process (MAP) was analyzed in \cite{feng2014potential}. Other inter\-/arrival times have also been considered. The Sparre Andersen risk process (perturbed by diffusion) with Erlang inter\-/arrival times was handled e.g.\ in \cite{gerber2005time,li2005gerber}. Some extensions have been considered lately. For example, a generalized discounted penalty function additionally takes into account the last minimum of the surplus before ruin in the analysis, see e.g.\ \cite{biffis2010anote,biffis2010generalization}.

This paper deals with another natural generalization of the classical risk process \eqref{Eq.Classical risk model} that included occasional and random gains (apart from the constant premium intensity), which are modeled by additional positive jumps. That is, the risk process \levymix takes the following form:
\begin{equation}\label{Eq.Levy process with positive phase-type jumps mixture model}
  \levymix = u+ ct +  \sum_{k=1}^{\fun[t]{\ppp}}\gpc[k] - \sum_{k=1}^{\fun[t]{\nppmix}}\mixc[,k].
\end{equation}
We assume that $\gpc[k] \stackrel{d}{=}\gpc$ are i.i.d.\ random variables. Moreover, the claims \mixc[,k] and gains \gpc[k] are independent of each other, where $\epsilon \in [0,1]$ is a parameter to be explained soon. Similarly, the gain arrival process \fun[t]{\ppp} and the claim arrival process \fun[t]{\nppmix} do not depend on each other either.

Such models have already been investigated. For example,
\citeauthor{xing2008time} \cite{xing2008time} studied the ruin time and the deficit at ruin for a risk process with double\-/sided jumps
when the upward jumps are phase\-/type distributed and the downward jumps have an arbitrary
distribution.
\citeauthor{jacobsen2005time} \cite{jacobsen2005time} derived a similar result for a more general risk model, the Markov\-/modulated diffusion risk model with two\-/sided jumps. More recently,
\citeauthor{zhang2010perturbed} \cite{zhang2010perturbed} and
\citeauthor{hua2013ruin} \cite{hua2013ruin} studied a GS function for the classical risk process \eqref{Eq.Levy process with positive phase-type jumps mixture model}.  They also studied the asymptotic estimate for the probability of ruin under heavy\-/tailed claims. Finally,
\citeauthor{kolkovska2015gerber} \cite{kolkovska2015gerber} analyzed the risk process \eqref{Eq.Levy process with positive phase-type jumps mixture model} perturbed by an $\alpha$-stable motion.

We combine another feature to our model: we include in our process some heavy\-/tailed risk
in a quite specific way. We assume that our claims' umbrella consists of two kinds of claims, the so\-/called light\-/tailed (attritional) claims modeled by a phase\-/type distribution and the heavy\-/tailed (large) claims with some subexponential distribution. In daily practice, insurance companies often have a small amount of heavy\-/tailed risks in their portfolio in addition to those smaller ones. We model it by its proportion $\epsilon>0$. Thus, the general claim size
distribution \mixcd can be expressed as follows:
\begin{equation}\label{Eq.Claim distribution function}
\mixcd:=(1-\epsilon)\ptcd +\epsilon \htcd.
\end{equation}
For simplicity, we assume that \ptcd is of phase\-/type. Heavy-tailed risk is very dangerous for the insurance industry. Indeed, insurers quite often may become insolvent or experience financial
distress due to large claims from catastrophic events and other large losses from financial investments. Such a study has become particularly relevant for insurance because of modern regulatory frameworks (e.g.\ EU Solvency) that require insurers to hold solvency capital so that the ruin probability is under control.

Following \cite{olvera2011transition} in this paper, we want to identify the impact of the heavy-tailed risk on the GS function, hence on the default probability,
when its proportion $\epsilon$ tends to $0$, thus is very small. One can also treat this result as an approximation procedure for the GS function of the process \eqref{Eq.Levy process with positive phase-type jumps mixture model} with phase\-/type upward gains and claims of the form \eqref{Eq.Claim distribution function}. We manage to prove that
\begin{equation*}
\gerbermixq = \gerberq + \epsilon h(u,\forceofinterest) + O(\epsilon^2),
\end{equation*}
where $h(u,\forceofinterest)$ can be found explicitly in terms of the model parameters.

The main idea is to use the so\-/called fluid embedding technique in which the upward jumps are placed at a slope one.
This path transformation makes the risk process upward\-/continuous similarly to the classical risk process.
Note that the original and transformed processes both get ruined over infinite time horizon or both survive but, obviously, this embedding affects the time scale. Since any upward jump has a phase\-/type distribution, which is the lifetime of some defective Markov chain, the modified risk process will move upwards with slope one as long as the Markov chain is still alive. Thus, we can treat any time of going upwards as the time of state changes of a certain Markov chain until its killing time. After this crucial path transformation, the risk process falls in the class of MAPs. We can then use the fluctuation theory for the latter developed in \cite{ivanovs2012occupation} based on scale matrices.
As a next step, we expand these scale matrices with respect to $\epsilon$ and derive our main result.

The paper is organized as follows.
In \Cref{Section: Presentation of the model}, we describe our model.
\Cref{Section: Fluid embedding and the Gerber-Shiu function for MAPs} focuses on fluid embedding of upward jumps.
In \Cref{Section: Perturbation of the model parameters}, we expand the scale matrices and other model parameters with respect to the above mentioned $\epsilon$.
\Cref{Section: Corrected approximation for the Gerber-Shiu function} gives the main result of this paper. Finally, we conclude in \Cref{Section: Conslusions}.

\section{Presentation of the model}\label{Section: Presentation of the model}
Let \levymix, $t \geq 0$, be a stochastic process defined on the probability space $(\Omega,\mathcal{F},\mathbb{P})$ that is equipped with a right\-/continuous complete filtration $\{\mathcal{F}_t\}_{t\geq 0 }$. We assume that \levymix represents the risk reserve process of an insurance company and we allow both for negative and positive claims, where the latter correspond to capital injections (gains). The gains \gpc arrive according to a Poisson process \ppp with rate \parrate and the claims \mixc arrive according to a Poisson process \nppmix with rate \narrate. Finally, we assume that the positive safety loading condition $c + \parrate \e \gpc - \narrate \e \mixc >0$ holds. Thus, \levymix takes the form \eqref{Eq.Levy process with positive phase-type jumps mixture model}, which was described briefly in \Cref{Section: Introduction}.

In our model, we assume that $\gpc[k] \stackrel{d}{=}\gpc$ are i.i.d.\ random variables with jump size distribution \pptcd of phase\-/type with \pphases phases and representation $\big( \pinp,\pintm \big)$; see \cite{asmussen-RP}. Moreover, the claims $\mixc[,k]\stackrel{d}{=}\mixc$ are i.i.d.\ random variables with jump size distribution \mixcd. Motivated by statistical analysis, which proposes that only a small fraction of the upper\-/order statistics is relevant for estimating tail probabilities, we consider that an arbitrary claim size \mixc is phase\-/type with probability $1-\epsilon$ and heavy\-/tailed with probability $\epsilon$, where $\epsilon \rightarrow 0$; see \cite{embrechts-MEE}. We assume that these phase\-/type claim sizes $\ptc[i] \stackrel{d}{=} \ptc$ and the heavy\-/tailed claim sizes $\htc[i] \stackrel{d}{=} \htc$ have both finite means, \mean[p] and \mean[h], respectively, are absolutely continuous, and we denote their corresponding distributions as \ptcd and \htcd.

\begin{remark}
Phase\-/type distributions are a computational vehicle for much of modern applied probability.
Recall that a distribution $F$ is said to be of phase\-/type if it is the same distribution
as the lifetime of some terminating Markov process starting from the vector $\boldsymbol{\alpha}$, having finitely many states, say $N$, and having time homogeneous transition rates with subintensity $\mathbf{T}$.
Namely,
$F(\phi)=1-\boldsymbol{\alpha}\exp\{\mathbf{T}\phi\} \unitcolumn$,
for $\phi\geq 0$, where \unitcolumn is a column\-/vector of ones.
This gives also the form of the density $f(\phi)$ of $F$ as
$f(\phi)=\boldsymbol{\alpha}\exp\{\mathbf{T}\phi\}\mathbf{t}$,
where
$\mathbf{t}=-\mathbf{T} \unitcolumn$.
Many well\-/known distributions are of phase\-/type; for example, the exponential distribution with intensity $\rho$ is PH in this case with $N=1$, $\mathbf{T}=\rho$, and $\boldsymbol{\alpha}=1$. Other more involved distributions that belong to the class of phase\-/types are the hyperexponential, the Erlang, and the Coxian; see \cite{asmussen-RP} for a detailed description.
\end{remark}

For the above mentioned model, we are interested in evaluating the so\-/called Gerber\-/Shiu or expected discounted penalty function. If the time to ruin is defined as $\ruintimemix := \inf\{t\geq 0 : \levymix<0\}$, then we have the following definition from \cite{gerber1998on}.

\begin{definition}
The Gerber\-/Shiu or expected discounted penalty function is defined as
  \begin{equation}
      \gerbermixq := \emeasure \left[ e^{-\forceofinterest \ruintimemix}\penaltyf{\deficitruinmix}{\surpluspriorruinmix} \indfun[\{\ruintimemix<\infty\}]\right],
  \end{equation}
 where \deficitruinmix is the {\it deficit at ruin}, \surpluspriorruinmix is the {\it surplus prior to ruin}, and $u\geq 0 $ is the initial surplus. The function  $\penaltyfunction : (0,\infty)^2 \rightarrow [0,\infty)$ is any bounded, measurable function, and $\forceofinterest \geq 0$, is the discount factor. Here the indicator function \indfun[\{\ruintimemix<\infty\}] emphasizes that the penalty is exercised only when ruin occurs.
\end{definition}

Assume from now on that $q>0$. In the absence of positive jumps, the model belongs to the class of spectrally negative L\'{e}vy\-/processes, i.e.\ it is a stochastic process with c\`{a}dl\`{a}g (right\-/continuous with left limits) sample paths, and stationary and independent increments, and the Gerber\-/Shiu function takes the form \cite[p.43]{kyprianou-GSRT}
\begin{equation}\label{Eq. Gerber-Shiu for spectrally negative Levy process}
  \gerbermixq =
            \narrate \int_{[0,\infty)} \int_{[0,\infty)} \penaltyf{y}{z} \densityresolventq \mixcd(z + dy)dz,
\end{equation}
where \densityresolventq is known as the \forceofinterest-resolvent or potential measure of \levymix killed on exiting the positive half\-/line.
In the case of existing upward jumps, we have to place our risk process into the framework of spectrally negative MAPs via fluid embedding. The spectral negativity means that our MAP has no positive jumps. In this way, we construct a new process \levymixnew governed by an environmental process \jumpprocess
for which we will apply an identity similar to \Cref{Eq. Gerber-Shiu for spectrally negative Levy process}. We explain in the next section how to construct an equivalent spectrally negative MAP and we summarize the respective necessary results for the evaluation of the Gerber\-/Shiu function.

\section{Fluid embedding and the Gerber-Shiu function for MAPs}\label{Section: Fluid embedding and the Gerber-Shiu function for MAPs}

When in a model the positive jumps are phase\-/type, a common approach is to spread them out as a succession of linear pieces of unit slope as  it was described e.g.\ in \cite{badescu2005surplus,breuer2011generalised}; see \Cref{Fig. Fluid embedding}. This procedure, which is called fluid embedding, requires adding supplementary states to the background process as many as the phases; namely \pphases in our model. By applying this technique, we construct an auxiliary spectrally negative MAP \levymixnew with an enlarged state space, say $\statespace = \{1,2,\dots,\pphases~+~1~\}$. The 1st state corresponds to the original risk process with negative jumps only. The other states correspond to the upward jumps transformed into lines. For the latter states, we have $d\levymixnew =dt$. In addition, we denote by \jumpprocess the right\-/continuous jump process that lives on \statespace.

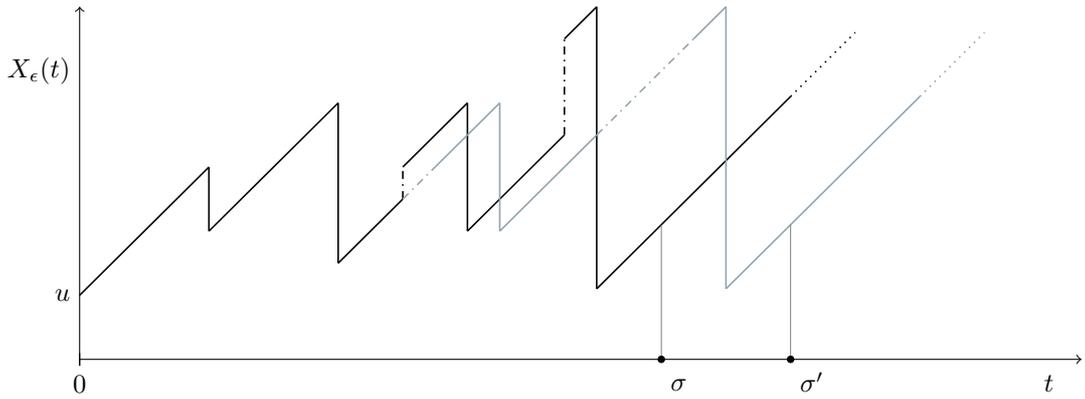
\begin{figure}
\centering
\begin{tikzpicture}[scale=0.85]
\draw[->] (0,0) -- (15.5,0);
\draw[->] (0,-0.1) -- (0,5.5);
\node [below] at (15 , -0.1) {$t$};
\node [left] at (0,4.5) {\levymix};
\node [left] at (0,1) {$u$};

\draw (0 , -0.1) -- (0 , 0.1);
\node [below] at (0 , -0.1) {$0$};

\draw [semithick] (0,1) -- (2,3);
\draw [semithick] (2,3) -- (2,2);
\draw [semithick] (2,2) -- (4,4);
\draw [semithick] (4,4) -- (4,1.5);
\draw [semithick] (4,1.5) -- (5,2.5);

\draw [semithick,dash dot] (5,2.5) -- (5,3);
\draw [semithick] (5,3) -- (6,4);
\draw [semithick] (6,4) -- (6,2);
\draw [semithick] (6,2) -- (7.5,3.5);
\draw [semithick,dash dot] (7.5,3.5) -- (7.5,5);
\draw [semithick] (7.5,5) -- (8,5.5);
\draw [semithick] (8,5.5) -- (8,1.1);
\draw [semithick] (8,1.1) -- (11,4.1);
\draw [semithick,dotted] (11,4.1) -- (12,5.1);
\draw [help lines] (9,2.1) -- (9,0);
\draw [fill] (9,0) circle [radius=0.05];
\node [right] at (9, -0.4) {$\sigma$};

\draw [semithick,phase,dash dot] (5,2.5) -- (5.5,3);
\draw [semithick,phase] (5.5,3) -- (6.5,4);
\draw [semithick,phase] (6.5,4) -- (6.5,2);
\draw [semithick,phase] (6.5,2) -- (8,3.5);
\draw [semithick,phase,dash dot] (8,3.5) -- (9.5,5);
\draw [semithick,phase] (9.5,5) -- (10,5.5);
\draw [semithick,phase] (10,5.5) -- (10,1.1);
\draw [semithick,phase] (10,1.1) -- (13,4.1);
\draw [semithick,phase,dotted] (13,4.1) -- (14,5.1);
\draw [help lines] (11,2.1) -- (11,0);
\draw [fill] (11,0) circle [radius=0.05];
\node [right] at (11, -0.35) {$\sigma'$};
\end{tikzpicture}
\caption{Sample path of a MAP with two\-/sided jumps (black line), where the positive phase\-/type jumps (dashed\-/dotted line) are replaced by linear stretches of slope 1, resulting in a spectrally\-/negative MAP with an augmented state\-/space (gray line). Fluid embedding shifts time $\sigma$ in the initial MAP to $\sigma'$ in the spectrally\-/negative MAP.}
\label{Fig. Fluid embedding}
\end{figure}

Note that for the process in \Cref{Eq.Levy process with positive phase-type jumps mixture model}, jumps are coming according to a superposition of the two independent Poisson processes \fun[t]{\ppp} and \fun[t]{\nppmix}. Hence, we obtain \Cref{Eq. Gerber-Shiu for spectrally negative Levy process} from the compensation formula (see \cite[Theorem~3.4]{kyprianou-GSRT}) because ruin can occur only when \fun[t]{\nppmix} increases, i.e.\  when a claim arrives, and thus only \narrate and \mixcd appear in the equation.
In addition, we know from \cite[p.39]{kyprianou-GSRT} that $\displaystyle \densityresolventq = \frac{1}{\forceofinterest dz} \pr_u \big( \levymix[e_\forceofinterest] \in dz, \inf_{w\leq e_\forceofinterest} \levymix[w] \geq 0 \big)$, where $e_\forceofinterest$ is an exponential random variable with parameter $\forceofinterest >0$ and independent of \levymix, while $\pr_u (\cdot) = \pr \big( \cdot | \levymix[0] = u \big)$. Note that for fixed $u$, the resolvent \densityresolventq should be treated as a Radon\-/Nikodym derivative of $\displaystyle \frac{1}{\forceofinterest dz} \pr_u \big( \levymix[e_\forceofinterest] \in dz, \inf_{w\leq e_\forceofinterest} \levymix[w] \geq 0 \big)$ with respect to the Lebesgue measure. Let now \densityresolventmatrixmixq be the \forceofinterest-resolvent of the process $\big( \levymixnew, \jumpprocess \big)$ constructed with fluid embedding killed on exiting from the positive half\-/line, i.e.
\begin{equation}\label{resolvent}
 \densityresolventmatrixmixq:=\int_0^\infty e^{- \forceofinterest t} \pr_u \big( \levymixnew \in dz, \inf_{w\leq t} \levymixnew[w] \geq 0, \jumpprocess \big)dt/dz,
\end{equation}
where for a random variable $L$ we write $\emeasure[u] [L; \jumpprocess ]$ to denote a
matrix with $ij$-th element
\begin{equation*}
\emeasure[u,i][L; \jumpprocess = j] = \e \big[ {L \indfun[\{\jumpprocess = j\}] \bigm\vert \levymixnew[0] = u, \jumpprocess[0] = i} \big],
\end{equation*}
and $\pr_u \big( A, \jumpprocess \big)=\emeasure[u] [\indfun[A]; \jumpprocess]$; see also \cite{feng2014potential}. Observe that all states in \statespace are fictitious for \levymix except for state 1. Moreover, stopping the process \levymix[e_\forceofinterest] at the exponential clock is possible only when this process has no jump at that moment $e_\forceofinterest$. That is, in particular $e_\forceofinterest$ cannot happen at the moment of a positive jump. Thus, if we want to stop the process \levymixnew at an exponential epoch $e_\forceofinterest$ at the same level as \levymix[e_\forceofinterest], then we should condition on the event that the state \jumpprocess[e_\forceofinterest] is equal to 1, since we take into account only the scenarios of stopping that happen at state 1. As a result, if $\constantfactor = 1/\pr_{u,1}\big( \jumpprocess[e_\forceofinterest]=1 \big)$, it follows that $\displaystyle \densityresolventq = \frac{1}{\forceofinterest dz} \pr_{u,1} \big( \levymixnew[e_\forceofinterest] \in dz| \jumpprocess[e_\forceofinterest]=1 \big) = \constantfactor \densityresolventmatrixmixq_{(1,1)}$ due to the lack of memory of the exponential distribution. Consequently, the GS function takes the next form after fluid embedding
\begin{equation}\label{Eq. Gerber-Shiu for spectrally negative Levy process simplified}
  \gerbermixq =
           \constantfactor \narrate \int_{[0,\infty)} \int_{[0,\infty)} \penaltyf{y}{z}  \densityresolventmatrixmixq_{(1,1)}
 \mixcd(z + dy)dz.
\end{equation}

Recall that the process $\big( \levymixnew, \jumpprocess \big)$ is a MAP if the pair $\big( \levymixnew[t+s]-\levymixnew,\jumpprocess[t+s] \big)$ given $\{\jumpprocess=i\}$ is independent of the natural history $\mathcal{F}_t$ up to time $t$ and has the same law as $\big( \levymixnew[s]-\levymixnew[0],\jumpprocess[s] \big)$ given $\{\jumpprocess[0]= i\}$, for all $s,t\geq 0$ and $i\in \statespace$. It is common to say that \levymixnew is an additive component and \jumpprocess is a background process
representing the environment. Importantly, MAPs have a very special structure, which we reveal in the following. It is immediate that \jumpprocess is a Markov chain. Furthermore, \levymixnew evolves as some L\'{e}vy process \levymixcase{i} while $\jumpprocess=i$.
In addition, a transition of \jumpprocess from $i$ to $j\neq i$ triggers a jump of \levymixnew. All the above components are independent.

Since \levymixnew has no positive jumps, it is uniquely characterized by the matrix exponent \mfmix, $s\geq 0$, which satisfies 
\begin{equation}\label{e4.matrix laplace exponent of the auxiliary MAP}
  \e\left[ e^{s \levymixnew}; \jumpprocess \right] = e^{\mfmix t}, \quad t\geq 0.
\end{equation}
See \cite{asmussen-RP} for all details.
If \ltptec and \lthtec are the Laplace\-/Stieltjes transforms (LSTs) of the stationary\-/excess claim size distributions \ptecd and \htecd (\eptc and \ehtc are the corresponding stationary\-/excess claim sizes), respectively, then from the above described fluid embedding, it follows that the matrix exponent \mfmix takes the form
\begin{equation}\label{e4.matrix exponent mixture model}
  \mfmix = \left[
           \begin{array}{c|c}
                \lapexpmix - \parrate        & \parrate \pinp \\ \hline
                \pexr                        & \pintm + s \im
           \end{array}
           \right],
\end{equation}
where
\begin{equation}\label{e4.laplace exponent of the mixture model}
  \lapexpmix = cs  - \narrate s \big((1-\epsilon) \mean[p] \ltptec + \epsilon \mean[h] \lthtec \big),
\end{equation}
and \im is the identity matrix of appropriate dimension.
Note that $\constantfactor =\big[ \forceofinterest \big( \forceofinterest \im-\mfmix[0] \big)^{-1}_{(1,1)}\big]^{-1}$.

\begin{remark}
  Observe that \lapexpmix corresponds to the L\'{e}vy exponent of \levymix in case the rate of the capital injections \parrate is equal to zero; namely when there exist only negative jumps. Moreover, note that \mfmix[0] is the transition rate matrix of \jumpprocess.
\end{remark}

From \Cref{Eq. Gerber-Shiu for spectrally negative Levy process simplified}, it is clear that we need to evaluate the \forceofinterest-resolvent measure for \levymixnew, which is known to be expressed in terms of the so\-/called \forceofinterest-scale matrices \cite{ivanovs2012occupation}. For this reason, we provide first its formal definition.

\begin{proposition}
  There exists a unique continuous function $\scalefunctionmix:[0,\infty) \rightarrow \mathds{R}^{(\pphases+1) \times (\pphases+1)}$, such that \scalematrixmix is invertible for all $x>0$,
  \begin{equation}\label{Eq. Definition of scale matrix for the mixture model}
    \int_0^{\infty} e^{-s x}\scalematrixmix dx = \big(\mfmix - \forceofinterestmatrix \big)^{-1},
  \end{equation}
  for all $s >\eta(\forceofinterest) = \max \{\Re(z):z \in \mathds{C}, \det\big(\mfmix - \forceofinterestmatrix \big) = 0\}$, where $\forceofinterestmatrix = \diag \{ \forceofinterest,0,\dots,0 \}$.
\end{proposition}
\begin{proof}
  Following \cite{breuer2011generalised}, we do not discount the time for the states $2,\dots,\pphases+1$ in order to account for the jump nature of the new phase\-/type upward movements. The result is then a direct application of \cite[Theorem~1]{ivanovs2012occupation}.
\end{proof}

For simplicity, we set $\mfmixq = \mfmix - \forceofinterestmatrix$. By using Theorem~4.8 of \cite{ivanovs}, one can immediately deduct that the equation $\det\mfmixq =0$ has $\pphases+1$ zeros with positive real part if $\forceofinterest>0$. Contrarily, for $\forceofinterest=0$ the number of zeros in $\mathds{C}^{\Re>0}$ are in number \pphases and zero is also a root. The latter is obvious since the first column of matrix \mfmix[0] is a linear combination of all its remaining columns. The solutions of $\det\mfmixq =0$ are called eigenvalues of the matrix \mfmixq. We assume here that these eigenvalues are simple; namely their multiplicities are all equal to one.

Following \cite{feng2014potential,ivanovs2014potential}, we define the matrix \leftsolutionmixq as the unique left solution to the matrix equation $\mfmixq[\leftsolutionmixq] = \zeromatrix$. This matrix is positive stable, i.e.\ all its eigenvalues have positive real part, and it is also diagonizable. More details regarding its construction will be provided in \Cref{Section: The perturbed matrix R}.

In the next lemma, we provide the exact expression for the \forceofinterest-resolvent measure of \levymixnew.

\begin{lemma}\label{Lemma. Resolvent ID}
  Fix $\forceofinterest> 0$. The \forceofinterest-resolvent measure for \levymixnew killed on exiting $[0,\infty)$ has density given by
  \begin{equation}\label{Eq.Resolvent matrix equation}
    \densityresolventmatrixmixq = \scalematrixmix[u] e^{-\leftsolutionmixq z} - \scalematrixmix[u-z],
  \end{equation}
  for all $u,z\geq 0$.
\end{lemma}
\begin{proof}
  The result is immediate from \cite[Theorem~1]{ivanovs2014potential}, by taking in Equation (14) $a=u$ and $x=z-u$, since we have here $\levymix[0] =u$. The above resolvent identity appears also in \cite{feng2014potential}.
\end{proof}

A direct consequence of \Cref{Lemma. Resolvent ID} is that the GS function in \Cref{Eq. Gerber-Shiu for spectrally negative Levy process simplified} further simplifies to
\begin{equation}\label{Eq. Gerber-Shiu for spectrally negative MAP}
  \gerbermixq =
            \constantfactor \narrate \int_{[0,\infty)} \int_{[0,\infty)} \penaltyf{y}{z} \Big(\scalematrixmix[u] e^{-\leftsolutionmixq z} - \scalematrixmix[u-z] \Big)_{(1,1)} \mixcd(z + dy)dz.
\end{equation}

Although the computational efforts have simplified significantly, it is still not easy to find a closed\-/form expression for the Gerber\-/Shiu function in \Cref{Eq. Gerber-Shiu for spectrally negative MAP}. The problem arises from the fact that \lthtec, which is the Laplace transform of a heavy\-/tailed distribution, appears in every element of the matrix $\big(\mfmix - \forceofinterestmatrix \big)^{-1}$ (specifically at the denominator, but we omit the details at this point), thus making Laplace inversion to evaluate the scale matrix difficult, if not impossible. However, if \lthtec was missing, then $\big(\mfmix - \forceofinterestmatrix \big)^{-1}$ would be a matrix of rational functions in $s$, and the evaluation of the scale matrix would be straightforward. Using this key idea, we explain in the next section how we can exploit that the claim size distribution is a mixture of \ptcd and \htcd and apply perturbation analysis to derive a series expansion for the Gerber\-/Shiu function. The goal is then to use the series expansion and construct accurate approximations for the measure under consideration.

\section{Perturbation of the model parameters}\label{Section: Perturbation of the model parameters}
Recall that the claim size distribution \mixcd is a mixture of a phase\-/type distribution \ptcd and a heavy\-/tailed one \htcd, i.e.\
$\mixcd(x) = (1-\epsilon) \ptcd(x) + \epsilon \htcd(x)$, $x \geq 0$. As mentioned in the previous section, if there were not any heavy\-/tailed jumps, for example if $\mixcd(x) = \ptcd(x) $, $x \geq 0$, then the algebraic calculations would reduce to the Laplace inversion of rational functions. Now, notice that
\begin{equation}\label{Eq. Perturbation of the claim size distribution}
\mixcd(x) = (1-\epsilon) \ptcd(x) + \epsilon \htcd(x)
= \ptcd(x) + \epsilon \big( \htcd(x) - \ptcd(x) \big),
\end{equation}
which means that the claim size distribution \mixcd can be regarded as perturbation of the distribution $\ptcd(x)$ by the function $\epsilon \big( \htcd(x) - \ptcd(x) \big)$. Consequently, we could calculate the Gerber\-/Shiu function with the claim size distribution $\ptcd(x) $ and then use perturbation analysis in order to derive a series expansion (or first order approximation) in $\epsilon$ for \gerbermixq.

Observe that perturbation on the claim size distribution affects also the scale matrix \scalematrixmix[u] as well as the matrix \leftsolutionmixq. Therefore, we first introduce the notation for the model with claim size distribution $\ptcd(x)$, which is obviously our phase\-/type base model, and then explain how to derive the parameters of the original model as perturbation of the parameters of the phase\-/type base model.

\subsection{Phase-type base model}\label{Section: Phase-type base model}
Setting $\epsilon = 0$ in \Cref{Eq. Perturbation of the claim size distribution} allows us to retrieve $\mixcd(x) = \ptcd(x)$. As we can see, the resulting mixture $\mixcd(x)$ is independent of the parameter $\epsilon$, and so are all the other parameters of this base model. Thus, to avoid overloading the notation, we omit the subscript ``0'' (which is a consequence of the fact that $\epsilon = 0$) from the parameters of the phase\-/type base model, which we call {\it base model} for simplicity. The Laplace exponent of the base model is
\begin{equation}
  \mfq = \left[
           \begin{array}{c|c}
                \lapexp - \parrate  -q    & \parrate \pinp \\ \hline
                \pexr                        & \pintm + s \im
           \end{array}
           \right],
\end{equation}
where
\begin{equation}
  \lapexp = cs - \narrate s \mean[p] \ltptec.
\end{equation}
Setting $\mxkelement{} = \narrate \big( \mean[h] \lthtec - \mean[p] \ltptec \big) $ for simplicity, we observe that $\lapexpmix = \lapexp - \epsilon s \mxkelement{}$ and that $\mfmixq = \mfq - \epsilon \mxk$, where
\begin{equation}
  \mxk = \left[
               \begin{array}{c|c}
                     s \mxkelement{}      & \vect{0}\\ \hline
                    \vect{0}                           & \vect{0}
               \end{array}
               \right].
\end{equation}

In the following sections, we explain how to find the scale matrix \scalematrixmix and the matrix \leftsolutionmixq as perturbation of the corresponding matrices \scalematrix and \leftsolutionq in the base model. We start with \scalematrixmix.

\subsection{The perturbed scale matrix}\label{Section: The perturbed scale matrix}

According to \Cref{Eq. Definition of scale matrix for the mixture model}, the scale matrix \scalematrixmix is defined as the Laplace inversion of the matrix $\big(\mfmixq\big)^{-1}$, the series expansion of which is given in the following line:
\begin{align}\label{Eq.Perturbed inverse}
  \big(\mfmixq\big)^{-1} &= \big(\mfq - \epsilon \mxk \big)^{-1}
                        = \sum_{n=0}^{\infty} \epsilon^n \Big(\big(\mfq\big)^{-1} \mxk \Big)^n \big(\mfq\big)^{-1}.
\end{align}

Note that the matrix \mxk is everywhere 0 except for its $(1,1)$ element. Thus, only the 1st row and 1st column of $\big( \mfq \big)^{-1}$ are involved in the series expansion \eqref{Eq.Perturbed inverse}. If \adjointmfq is the adjoint matrix of \mfq, then it is known that $\big( \mfq \big)^{-1} = \adjointmfq / \det \mfq$. Let \adjointmfelementq be the $(i,j)$ element of the matrix \adjointmfq. Hence \Cref{Eq.Perturbed inverse} takes the form
\begin{align}\label{Eq.Series expansion perturbed inverse}
\big(\mfmixq\big)^{-1} =& \big(\mfq\big)^{-1}
+ \sum_{n=1}^\infty \epsilon^n s^n \mxkelement{n} \frac{ \Big( \det\big( \pintm + s \im \big) \Big)^{n-1}}{ \Big( \det \mfq \Big)^{n+1}} \notag \\
& \times
\begin{pmatrix}
\adjointmfelementq[11]\adjointmfelementq[11]	&\adjointmfelementq[11]\adjointmfelementq[12]	&\cdots	&\adjointmfelementq[11]\adjointmfelementq[1,\pphases+1]\\
\adjointmfelementq[21]\adjointmfelementq[11]	&\adjointmfelementq[21]\adjointmfelementq[12]	&\cdots	&\adjointmfelementq[21]\adjointmfelementq[1,\pphases+1]\\
\vdots	&\vdots	&\ddots	&\vdots \\
\adjointmfelementq[\pphases+1,1]\adjointmfelementq[11]	&\adjointmfelementq[\pphases+1,1]\adjointmfelementq[12]	&\cdots	&\adjointmfelementq[\pphases+1,1]\adjointmfelementq[1,\pphases+1]\\
\end{pmatrix},
\end{align}
where $\adjointmfelementq[11] = \det\big( \pintm + s \im  \big)$. Recall now that we only need the $(1,1)$ element of the matrix \densityresolventmatrixmixq. Therefore, according to \Cref{Eq.Resolvent matrix equation}, we just need to find the Laplace inversion of the first row of the matrix exponent $\big(\mfmixq\big)^{-1}$. In particular, its $(1,1)$ element has the series representation
\begin{equation}\label{Eq.Series expansion for the first element of matrix exponent}
 \big(\mfmixq\big)^{-1}_{(1,1)}
 = \big(\mfq\big)^{-1}_{(1,1)}
+ \sum_{n=1}^\infty \epsilon^n \mxkelement{n} \left( \frac{  s \det\big( \pintm + s \im  \big) }{  \det \mfq } \right)^n \big(\mfq\big)^{-1}_{(1,1)},
\end{equation}
given that $\big(\mfq\big)^{-1}_{(1,1)} = \det\big( \pintm + s \im \big) / \det \mfq$.

From \cite[Theorem~5]{ivanovs2012occupation} and the subsequent comment, it follows that the scale matrices \scalematrixmix and \scalematrix are entry\-/wise differentiable. Hence, by following the reasoning in \cite[Section~4.1]{kyprianou-GSRT}, we let \scalematrixmix[dx] and \scalematrix[dx] be the matrices such that $ \int_0^{\infty} e^{-s x}\scalematrixmix[dx] = s \big(\mfmixq\big)^{-1}$ and $ \int_0^{\infty} e^{-s x}\scalematrix[dx] = s \big(\mfq\big)^{-1}$, respectively. Moreover, taking into account the binomial identity $\mxkelement{}= \big( \narrate \big)^n \sum_{m=0}^n \binom{n}{m} \mean[h]^{n-m} (-\mean[p])^m \big( \lthtec \big)^{n-m} \big(\ltptec \big)^m$, we apply Laplace inversion to \Cref{Eq.Series expansion for the first element of matrix exponent} to obtain
\begin{align}\label{Eq.Series expansion for the first element of the scale matrix}
\big( \scalematrixmix[dx] \big)_{(1,1)}
=\ &\big( \scalematrix[dx] \big)_{(1,1)} \notag \\
&+ \sum_{n=1}^\infty \epsilon^n \big( \narrate \big)^n \sum_{m=0}^n \binom{n}{m} \mean[h]^{n-m} (-\mean[p])^m L_{n-m,m}(dx) * \big( \scalematrix[dx] \big)_{(1,1)}^{*(n+1)},
\end{align}
where $\eptc[i] \stackrel{d}{=} \eptc$ and $\ehtc[i] \stackrel{d}{=} \ehtc$ are respectively the stationary\-/excess phase\-/type and heavy\-/tailed claim sizes, $L_{s,r}(dx) = \pr ( \ehtc[1] + \dots + \ehtc[s] + \eptc[1] + \dots + \eptc[r] \in dx)$, the symbol $*$ denotes the convolution between functions, and $\big( \scalematrix[dx] \big)_{(1,1)}^{*n}$ is the $n$th convolution of $\big( \scalematrix[dx] \big)_{(1,1)}$ with itself. Any other element of the first row of the matrix expansion \eqref{Eq.Series expansion perturbed inverse} takes the form
\begin{equation}\label{Eq.Series expansion for any other element of matrix exponent}
 \big(\mfmixq\big)^{-1}_{(1,j)}
 = \big(\mfq\big)^{-1}_{(1,j)}
+ \sum_{n=1}^\infty \epsilon^n \mxkelement{n} \left( \frac{  s \det\big( \pintm + s \im \big) }{  \det \mfq } \right)^n \big(\mfq\big)^{-1}_{(1,j)},
\end{equation}
and consequently, Laplace inversion gives
\begin{align}\label{Eq.Series expansion for any other element of the scale matrix}
\big( \scalematrixmix[dx] \big)_{(1,j)}
 &=\big( \scalematrix[dx] \big)_{(1,j)} \notag \\
+& \sum_{n=1}^\infty \epsilon^n \big( \narrate \big)^n \sum_{m=0}^n \binom{n}{m} \mean[h]^{n-m} (-\mean[p])^m L_{n-m,m}(dx) * \big( \scalematrix[dx] \big)_{(1,1)}^{*n}* \big( \scalematrix[dx] \big)_{(1,j)}.
\end{align}
Given \Cref{Eq.Series expansion for the first element of the scale matrix,Eq.Series expansion for any other element of the scale matrix}, it is easy to find the elements of the scale matrix \scalematrixmix that are important for the evaluation of the Gerber\-/Shiu function. In the next section, we turn our attention to the matrix \leftsolutionmixq, which relates to the eigenvalues with non\-/negative real part (and their eigenvectors) of \mfmixq.

\begin{remark}
Observe that we denote by $f(x)*g(x)$ the convolution between the functions $f$ and $g$ instead of the traditional choice $(f*g)(x)$, i.e.\ $f(x)*g(x) = \int_0^x f(t)g(x-t)dt$. This convention is used in order to avoid the introduction of new notation, e.g.\ it allows us to use interchangeably the convolutions $L_{n-m,m}(x) * \scalematrix[dx] $ and $L_{n-m,m}(dx) *  \scalematrix[x] $, while the same principle applies to convolutions of similar form.
\end{remark}

\subsection{The perturbed matrix \leftsolutionmixq }\label{Section: The perturbed matrix R}
Since \leftsolutionmixq is diagonizable, this means that $\leftsolutionmixq = \big( \leftkernelmixq \big)^{-1} \diagonalmatrixmixq \leftkernelmixq$ for some invertible matrix \leftkernelmixq and a real diagonal matrix \diagonalmatrixmixq. More precisely, the diagonal elements of \diagonalmatrixmixq are exactly the $\pphases+1$ solutions of $\det \mfmixq =0$ with non\-/negative real part; let us denote them as \simplerootq[\epsilon,i], $i=1,\dots,\pphases+1$ (we order them w.l.o.g.\ according to their real part). Then, with similar arguments as in \cite[Corollary~3.6]{ivanovs}, it can easily be shown that the $i$th row of \leftkernelmixq belongs to the (left) nullspace of \mfmixq[{\simplerootq[\epsilon,i]}], $i \in \statespace$.

Thus, in order to construct the matrix \leftsolutionmixq, we need to find first the eigenvalues of \mfmixq, i.e.\ the solutions of the equation $\det \mfmixq =0$ with non\-/negative real part. We calculate these eigenvalues in \Cref{Section: The perturbed eigenvalues with positive real part}. Afterwards, in \Cref{Section: The perturbed left eigenvectors}, we explain how to construct the matrix \leftkernelmixq.

\subsubsection{The perturbed eigenvalues with positive real part}\label{Section: The perturbed eigenvalues with positive real part}

We can easily calculate
\begin{equation}\label{Equation: Determinant of the perturbed matrix exponent}
\det \mfmixq = \det \mfq - \epsilon s \mxkelement{} \det\big( \pintm + s \im \big),
\end{equation}
and it is rather obvious that $0$ is an eigenvalue (simple) of \mfmixq only when $\forceofinterest=0$, which is excluded for now. Let now \simplerootq[i] be any other simple eigenvalue of \mfq with positive real part. Assuming that both \mfq and \mxk are analytic in the neighbourhood of \simplerootq[i], it is known from \cite[Chapter \textsc{Two} \S1]{kato-PTLO} that there exists a unique eigenvalue \simplerootq[\epsilon,i] of \mfmixq with expansion
\begin{equation}
\simplerootq[\epsilon,i] = \simplerootq[i] + \epsilon \simplerootperturbationq[i,1] + \epsilon^2 \simplerootperturbationq[i,2] + \dots
\end{equation}
Thus, if $\diagonalmatrixq = \diag \big( \simplerootq[i] \big)_{i \in \statespace}$, then \diagonalmatrixmixq can be written as perturbation of \diagonalmatrixq as follows:
\begin{equation}
\diagonalmatrixmixq = \diagonalmatrixq + \epsilon \diagonalmatrixperturbationq[1] + \epsilon^2 \diagonalmatrixperturbationq[2] + \dots
\end{equation}

\begin{remark}\label{Remark. Recursive scheme}
The exact value of the coefficient \simplerootperturbationq[i,1] has already been evaluated in \cite[Corollary~A.11]{vatamidou} and it is equal to
\begin{equation}
\simplerootperturbationq[i,1] = \frac{\mxkelement[{\simplerootq[i]}]{} \det\big( \pintm + \simplerootq[i] \im \big)}{\frac{d}{ds}\det \mfq \Bigm\vert_{s=\simplerootq[i]}}.
\end{equation}
The rest of the coefficients \simplerootperturbationq[i,j], $j=2,3,\dots$, can be found recursively by using the Taylor series of \mfq and \mxk around \simplerootq[i] and equating all the coefficients of $\epsilon^n$ ($n=2,3,\dots$) in $\det \mfmixq[{\simplerootq[\epsilon,i]}] =0$ to 0.
\end{remark}

\subsubsection{The perturbed left eigenvectors }\label{Section: The perturbed left eigenvectors}

Using similar arguments as in \cite[Theorem~A.4]{vatamidou}, we can show that the rows of \leftkernelmixq can be defined with the aid of the adjoint matrix of \mfmixq. In particular, if \adjointmfmixq is the adjoint of \mfmixq, then the $i$th row of \leftkernelmixq, $i \in \statespace$, is equal to any row of the matrix \adjointmfmixq[{\simplerootq[\epsilon,i]}] that is not identically equal to zero. Thus, as a first step, we identify in the next theorem the exact form of a non\-/zero row of the matrix \adjointmfmixq.

\begin{theorem}\label{Theorem: Non-zero row of the mix adjoint matrix}
There exists a row of the matrix \adjointmfmixq that is not identically equal to zero and takes the form
\begin{equation}\label{Eq.Exact form of eigenvalues}
\Big( \begin{array}{c;{2pt/2pt}c}
\det\big( \pintm + s \im  \big) & -\parrate \pinp \adj\big(  \pintm + s \im \big)\end{array} \Big).
\end{equation}
If $\forceofinterest=0$ then $\det \mfmix[0] = 0$, which means that 0 is an eigenvalue of \mfmix. Therefore, for $q=s=0$, \Cref{Eq.Exact form of eigenvalues} simplifies to
\begin{equation}\label{Eq.Exact form of eigenvalues simplified}
\Big( \begin{array}{c;{2pt/2pt}c}
\det \pintm  & -\parrate \pinp \adj \pintm \end{array} \Big)
=
\det \pintm  \cdot \Big( \begin{array}{c;{2pt/2pt}c}
1 & -\parrate \pinp \big( \pintm \big)^{-1} \end{array} \Big).
\end{equation}
\end{theorem}

\begin{proof}
Recall that to find the $(i,j)$ element of the adjoint matrix \adjointmfq, we calculate the determinant that occurs when we delete the $j$\-/row and the $i$\-/column of the initial matrix. Since the matrices \mfq and \mfmixq differ only on the $(1,1)$ element, the matrices \adjointmfq and \adjointmfmixq are equal on the 1st row and  1st column; i.e.\ if we delete the $(1,1)$ element either by deleting the 1st row or column, then the remaining minors are equal. Another key observation is that deleting the $(j,i)$ row\-/column combination in the matrix \mfmixq, with $j,i \geq 2$, it means that we delete the $(j-1,i-1)$ row\-/column combination in the matrix $\pintm + s \im $. Therefore, it is easy to verify that
\begin{equation}
  \adjointmfmixq = \adjointmfq - \epsilon\adjointmxk,
\end{equation}
where
\begin{equation}\label{Eq. Definition of the adjoint matrix of K}
  \adjointmxk =  \left[
               \begin{array}{c|c}
                   0                       & \vect{0}\\ \hline
                    \vect{0}               &  s \mxkelement{}\adj\big(  \pintm + s \im \big)
               \end{array}
               \right].
\end{equation}

Any row of \adjointmfmixq that is not identically equal to zero is a suitable candidate. Therefore, we can choose w.l.o.g.\ its 1st row because the matrix \adjointmxk has the 1st row equal to zero and this simplifies the calculations. As a result, we need to find the 1st row of \adjointmfq. To do so, we use \eqref{Equation: Determinant of the perturbed matrix exponent} and the identity $\mfmixq\adjointmfmixq = \det \mfmixq \im$. We obtain
\begin{align}
  \Big(  \mfq - \epsilon\mxk\Big) \Big(  \adjointmfq - \epsilon\adjointmxk\Big)
                        &= \det \mfq \im - \epsilon s \mxkelement{} \det\big( \pintm + s \im \big) \im \notag \\
                        &\Rightarrow \notag \\
  \mfq\adjointmxk + \mxk\adjointmfq &= s \mxkelement{} \det\big( \pintm + s \im  \big) \im, \label{Eq.Identity between matrix, adjoint, and determinant}
\end{align}
since $\mxk\adjointmxk = \vect{0}$. The above equation means that only the diagonal elements of the matrix $\mfq\adjointmxk + \mxk\adjointmfq$ are not identically equal to zero. Observe that we only need the 1st row of \adjointmfq. Obviously, the $(1,1)$ element of \adjointmfq is equal to $\det\big( \pintm + s \im \big)$. To find the remaining elements, we find the 1st row of $\mfq\adjointmxk$ as
\begin{equation*}
\Big( \begin{array}{c;{2pt/2pt}c} 0 & \parrate \pinp s \mxkelement{} \adj\big(  \pintm + s \im \big)\end{array} \Big).
\end{equation*}
and then use the identity \eqref{Eq.Identity between matrix, adjoint, and determinant} to retrieve \Cref{Eq.Exact form of eigenvalues}.
\end{proof}

As explained in the proof of \Cref{Theorem: Non-zero row of the mix adjoint matrix}, the vector \eqref{Eq.Exact form of eigenvalues} is also any row --not identically equal to zero-- of \adjointmfq. If \leftkernelq is the respective matrix of the phase\-/type base model, then its $i$th row is equal to
\begin{equation}
\Big( \begin{array}{c;{2pt/2pt}c} \det\big( \pintm + \simplerootq[i] \im \big) & -\parrate \pinp \adj\big(  \pintm + \simplerootq[i] \im  \big)\end{array} \Big).
\end{equation}
The next theorem states that we can express \leftkernelmixq as a perturbation of \leftkernelq. We denote by $\mathbf{A}_{i \bullet}$ and $\mathbf{A}_{\bullet j}$ the $i$th row and $j$th column of a matrix $\mathbf{A}$, and by \unitvector and \unitmatrix the row vector and matrix respectively with all elements equal to 1.

\begin{theorem}
The matrix \leftkernelmixq admits a series expansion in $\epsilon$ with first term the matrix \leftkernelq, i.e.
\begin{align*}
\big( \leftkernelmixq \big)_{i \bullet}
&= \big( \leftkernelq \big)_{i \bullet} + \epsilon \big ( \leftkernelperturbationq[1] \big)_{i \bullet} + O(\epsilon^2 \unitvector),\\
\intertext{where}
\big ( \leftkernelperturbationq[1] \big)_{i \bullet}
&= \simplerootperturbationq[i,1] {\frac{d}{ds} \Big( \begin{array}{c;{2pt/2pt}c}
\det\big( \pintm + s \im \big) & -\parrate \pinp \adj\big(  \pintm + s \im \big)\end{array} \Big) \Bigm\vert_{s=\simplerootq[i]}}.
\end{align*}
If $q=0$, then the first row of the matrix \leftkernelperturbationq[1] is identically equal to zero.
\end{theorem}

\begin{proof}
Observe that the 1st element of the vector \eqref{Eq.Exact form of eigenvalues} is a polynomial in $s$ of degree \pphases, while all its other elements are polynomials in $s$ of degree $\pphases-1$. Therefore, the result is immediate by substituting in these polynomials $s = \simplerootq[i] + \epsilon \simplerootperturbationq[i,1] + \epsilon^2 \simplerootperturbationq[i,2] + \dots$ and finding the coefficient of $\epsilon$.
\end{proof}

Deriving explicitly the series expansion of \leftkernelmixq is possible if one follows a similar recursive scheme to the one described in \Cref{Remark. Recursive scheme}. Moreover, we can find
\begin{equation}
\big( \leftkernelmixq \big)^{-1} = \big( \leftkernelq \big)^{-1} - \epsilon \big( \leftkernelq \big)^{-1} \leftkernelperturbationq[1] \big( \leftkernelq \big)^{-1} + O(\epsilon^2 \unitmatrix).
\end{equation}

Finally, we calculate in the next section the matrix exponential $e^{-\leftsolutionmixq z}$ that appears in \Cref{Eq. Gerber-Shiu for spectrally negative MAP}.

\subsubsection{The perturbed matrix exponential}\label{Section: The perturbed matrix exponential}

\begin{proposition}\label{Proposition. The perturned exponential matrix}
The series expansion of the matrix exponent $e^{-\leftsolutionmixq z}$ in \Cref{Eq. Gerber-Shiu for spectrally negative MAP} takes the form
\begin{align*}
e^{-\leftsolutionmixq z} =& \ e^{-\leftsolutionq z} + \epsilon \Big( \big( \leftkernelq \big)^{-1} e^{- \diagonalmatrixq z} \leftkernelperturbationq[1] - z \big( \leftkernelq \big)^{-1} e^{- \diagonalmatrixq z} \diagonalmatrixperturbationq[1] \leftkernelq - \big( \leftkernelq \big)^{-1} \leftkernelperturbationq[1] e^{-\leftsolutionq z} \Big) \notag \\
&+ O(\epsilon^2 \unitmatrix).
\end{align*}
\end{proposition}

\begin{proof}
Using the definition of the matrix exponent, i.e.\ $e^{\vect{A}} = \sum_{k=0}^{+\infty} \vect{A}^k /k!$ for a square matrix \vect{A}, we can find that $\displaystyle e^{-\diagonalmatrixmixq z} = e^{- \diagonalmatrixq z} - \epsilon  e^{- \diagonalmatrixq z} \diagonalmatrixperturbationq[1] z + O(\epsilon^2 \unitmatrix) $ because \diagonalmatrixq and \diagonalmatrixperturbationq[1] are diagonal matrices. Thus,
\begin{align*}
e^{-\leftsolutionmixq z}
=& \big( \leftkernelmixq \big)^{-1} e^{-\diagonalmatrixmixq z} \leftkernelmixq\\
=& \Big(  \big( \leftkernelq \big)^{-1} - \epsilon \big( \leftkernelq \big)^{-1} \leftkernelperturbationq[1] \big( \leftkernelq \big)^{-1} + O(\epsilon^2 \unitmatrix) \Big)
e^{-\big( \diagonalmatrixq + \epsilon \diagonalmatrixperturbationq[1] + O(\epsilon^2 \unitmatrix) \big) z} \\
& \times \Big(  \leftkernelq  + \epsilon \leftkernelperturbationq[1]+ O(\epsilon^2 \unitmatrix) \Big)\\
=&\Big(  \big( \leftkernelq \big)^{-1} - \epsilon \big( \leftkernelq \big)^{-1} \leftkernelperturbationq[1] \big( \leftkernelq \big)^{-1} + O(\epsilon^2 \unitmatrix) \Big)
\Big( e^{- \diagonalmatrixq z} - \epsilon  e^{- \diagonalmatrixq z} \diagonalmatrixperturbationq[1] z + O(\epsilon^2 \unitmatrix)  \Big)\\
& \times \Big(  \leftkernelq  + \epsilon \leftkernelperturbationq[1]+ O(\epsilon^2 \unitmatrix) \Big)\\
=& e^{-\leftsolutionq z} + \epsilon \Big( \big( \leftkernelq \big)^{-1} e^{- \diagonalmatrixq z} \leftkernelperturbationq[1] - z \big( \leftkernelq \big)^{-1} e^{- \diagonalmatrixq z} \diagonalmatrixperturbationq[1] \leftkernelq - \big( \leftkernelq \big)^{-1} \leftkernelperturbationq[1] e^{-\leftsolutionq z} \Big)
+ O(\epsilon^2 \unitmatrix).
\end{align*}
\end{proof}

\begin{corollary}\label{Corollary. Resolvent ID perturbed}
  The series expansion of the \forceofinterest-resolvent measure $\big( \densityresolventmatrixmixq \big)_{(1,1)}$ in \Cref{Lemma. Resolvent ID} takes the form
  \begin{align}\label{Eq.Resolvent matrix equation perturbed}
   \big( \densityresolventmatrixmixq \big)_{(1,1)}
   =& \big( \densityresolventmatrixq \big)_{(1,1)} + \epsilon \correctionresolventq + O(\epsilon^2),
  \end{align}
for all $u,z\geq 0$, where
   \begin{align}
   \correctionresolventq
    =& \narrate \big( \mean[h] \htecd[u] -\mean[p] \ptecd[u] \big)* \big( \scalematrix[du] \big)_{(1,1)} *\big( \scalematrix[du] \big)_{1\bullet} \cdot \Big( e^{-\leftsolutionq z} \Big)_{\bullet 1} \notag \\
    &+ \big( \scalematrix[u] \big)_{1\bullet} \cdot \Big( \big( \leftkernelq \big)^{-1} e^{- \diagonalmatrixq z} \leftkernelperturbationq[1] - z \big( \leftkernelq \big)^{-1} e^{- \diagonalmatrixq z} \diagonalmatrixperturbationq[1] \leftkernelq - \big( \leftkernelq \big)^{-1} \leftkernelperturbationq[1] e^{-\leftsolutionq z} \Big)_{\bullet 1} \notag \\
    &- \narrate \int_0^{u-z}\big( \mean[h] \htecd[u-z-x] -\mean[p] \ptecd[u-z-x] \big) \big( \scalematrix[dx] \big)_{(1,1)}^{*2}.\label{Eq. The correction term of the resolvent matrix}
  \end{align}
\end{corollary}
\begin{proof}

From \Cref{Eq.Series expansion for the first element of the scale matrix,Eq.Series expansion for any other element of the scale matrix}, we find that the 1st row of the matrix \scalematrixmix is given by
\begin{equation*}
\big( \scalematrixmix \big)_{1\bullet}
=\ \big( \scalematrix \big)_{1\bullet}
+ \epsilon \narrate \big( \mean[h] \htecd -\mean[p] \ptecd \big)* \big( \scalematrix[dx] \big)_{(1,1)}*\big( \scalematrix[dx] \big)_{1\bullet} + O(\epsilon^2 \unitmatrix).
\end{equation*}
Thus, to retrieve \Cref{Eq. The correction term of the resolvent matrix}, we need to combine the above formula with \Cref{Lemma. Resolvent ID} and \Cref{Proposition. The perturned exponential matrix}, and identify the coefficient of $\epsilon$.
\end{proof}

\section{Corrected approximation for the Gerber-Shiu function}\label{Section: Corrected approximation for the Gerber-Shiu function}

Combining the previous results produces the following expansion for the Gerber\-/Shiu function.

\begin{theorem}\label{Theorem. Series expansion of the GS function}
Let $q>0$. The GS function \gerbermixq in \Cref{Eq. Gerber-Shiu for spectrally negative MAP} has a series expansion as follows
\begin{align*}
  \gerbermixq =&\ \gerberq \\
  &+ \epsilon \constantfactor \narrate \Bigg(
             \int_{[0,\infty)} \int_{[0,\infty)} \penaltyf{y}{z}  \densityresolventmatrixq_{(1,1)} \Big(
 \htcd(z + dy) - \ptcd(z + dy) \Big) dz \\
  &+
             \int_{[0,\infty)} \int_{[0,\infty)} \penaltyf{y}{z}  \correctionresolventq
 \ptcd(z + dy)   dz \Bigg)
 + O(\epsilon^2),
\end{align*}
where \gerberq and $\densityresolventmatrixq_{(1,1)}$ are evaluated respectively through \Cref{Eq. Gerber-Shiu for spectrally negative Levy process simplified,Eq.Resolvent matrix equation} with $\epsilon=0$.
\end{theorem}

Note that \gerberq is the GS function for the base model without the heavy\-/tailed component \htcd and having only phase\-/type upward and downwards jumps. Moreover, the error that this approximation makes is of order $O(\epsilon)$. By keeping only the two first terms of the series expansion of \gerbermixq in \Cref{Theorem. Series expansion of the GS function}, we can define an approximation for the GS function that makes an error of order $O(\epsilon^2)$. We call it {\it corrected phase\-/type approximation} because the $\epsilon$-order term corrects the behavior of \gerberq, which is by itself a phase\-/type approximation for \gerbermixq.

\begin{approximation}\label{Approximation}
For $\forceofinterest >0$, the corrected phase\-/type approximation for \gerbermixq is defined as
\begin{align*}
  \approxgerbermixq :=&\ \gerberq \\
  &+ \epsilon \constantfactor \narrate \Bigg(
             \int_{[0,\infty)} \int_{[0,\infty)} \penaltyf{y}{z}  \densityresolventmatrixq_{(1,1)} \Big(
 \htcd(z + dy) - \ptcd(z + dy) \Big) dz \\
  &+
             \int_{[0,\infty)} \int_{[0,\infty)} \penaltyf{y}{z}  \correctionresolventq
 \ptcd(z + dy)   dz \Bigg),
\end{align*}
where $\gerberq = \displaystyle \narrate
             \int_{[0,\infty)} \int_{[0,\infty)} \penaltyf{y}{z}  \densityresolventmatrixq_{(1,1)} \ptcd(z + dy) dz $.
\end{approximation}

Observe that one can construct even more accurate approximations for \gerbermixq by simply keeping more terms in the series expansion of \Cref{Theorem. Series expansion of the GS function}. More precisely, an approximation including up to the $\epsilon^n$-order term gives an error of $O(\epsilon^{n+1})$. However, calculating higher order terms is not always easy. The difficulty mainly comes from the evaluation of the series expansion of the matrix $e^{-\leftsolutionmixq z}$ because the scale matrix \scalematrixmix[u] has an explicit series expansion. In certain models where the series expansion of $e^{-\leftsolutionmixq z}$ can be calculated explicitly, then it is also possible to derive the whole series expansion for \gerbermixq as well. One such particular example is the ruin probability in the classical Cram{\'e}r\-/Lundberg.

\subsection{Application -- Ruin probability in the Cram{\'e}r-Lundberg risk model}
In this section, we derive the corrected phase\-/type approximation for the ruin probability of the classical Cram{\'e}r\-/Lundberg risk model. Since there are no positive jumps, $\parrate = \pphases = 0$, $\levymixnew \equiv \levymix$, and $\constantfactor=1$. Moreover, we must take $\penaltyfunction \equiv 1$ and $\forceofinterest=0$. In fact, to apply \Cref{Approximation}, we should take $\forceofinterest \downarrow 0$ and treat both sides of this approximation in the limiting sense. Note that this procedure is correct since both sides of this identity converge as $\forceofinterest \downarrow 0$.
According to \Cref{Section: Fluid embedding and the Gerber-Shiu function for MAPs}, the matrix \mfmix, which is actually equal to the scalar \lapexpmix in \Cref{e4.laplace exponent of the mixture model}, has zero as its only single eigenvalue and thus $e^{-\leftsolutionmixq z} = 1$, which makes the middle term in \eqref{Eq. The correction term of the resolvent matrix} vanish. Finally, the positive safety loading condition takes the form $\narrate (1-\epsilon)\mean[p] + \narrate \epsilon\mean[h] <c$ for the mixture model and $\narrate \mean[p] /c <1$ for the base model.

For the base model, we calculate the scale matrix/function through
\begin{align*}
   \big( \mf \big)^{-1} = \frac{1}{cs} \cdot \frac{1}{1- \frac{\narrate \mean[p]}{c} \ltptec }
   = \frac{1}{cs} \sum_{n=0}^{+\infty} \left( \frac{\narrate \mean[p]}{c} \right)^n \big( \ltptec \big)^n
   \Rightarrow
   \scalematrixsimple = \frac{1}{c} \sum_{n=0}^{+\infty} \left( \frac{\narrate \mean[p]}{c} \right)^n \ptecconv{n},
\end{align*}
while the resolvent is equal to
\begin{align*}
    \densityresolventmatrix
    = \scalematrixsimple[u] - \scalematrixsimple[u-z]
    =  \frac{1}{c} \sum_{n=0}^{+\infty} \left( \frac{\narrate \mean[p]}{c} \right)^n \Big( \ptecconv[u]{n} - \ptecconv[u-z]{n} \Big).
\end{align*}
If $M=-\inf_{0 \leq t < \infty} \big( \levy - \levy[0] \big)$ and \ruinprobability is the ruin probability of the base model, then the Pollaczeck\-/Khinchine formula gives
\begin{gather*}
    \pr (M \leq u) = 1 - \ruinprobability
    = \left( 1 - \frac{\narrate \mean[p]}{c} \right) \sum_{n=0}^{+\infty} \left( \frac{\narrate \mean[p]}{c} \right)^n \ptecconv[u]{n}
    = \big( c - \narrate \mean[p] \big) \scalematrixsimple[u]\\
    \Rightarrow \\
    \densityresolventmatrix = \frac{1}{c - \narrate \mean[p]} \big( \pr (M \leq u) - \pr (M \leq u-z) \big).
\end{gather*}
In addition, we calculate
   \begin{align*}
   \correctionresolventq
    =& v^{(0)}(u,z)=\narrate \big( \mean[h] \htecd[u] -\mean[p] \ptecd[u] \big)* \big( \scalematrixsimple[du] \big)^{*2}  \\
    &- \narrate \int_0^{u-z}\big( \mean[h] \htecd[u-z-x] -\mean[p] \ptecd[u-z-x] \big) \big( \scalematrixsimple[dx] \big)^{*2}\\
    =& \frac{\narrate}{(c - \narrate \mean[p])^2} \big( \mean[h] \htecd[u] -\mean[p] \ptecd[u] \big)* \big( \pr(M \in du) \big)^{*2}\\
    &- \frac{\narrate}{(c - \narrate \mean[p])^2} \int_0^{u-z}\big( \mean[h] \htecd[u-z-x] -\mean[p] \ptecd[u-z-x] \big) \big( \pr(M \in dx) \big)^{*2}.
  \end{align*}

Noting that $\int_0^\infty \htcd(z + dy) = \mean[h]\htecd[dz]/dz$ and $\int_0^\infty \ptcd(z + dy) = \mean[p]\ptecd[dz]/dz$ and letting $M^* \equalindistribution M$, $\ptc^{e*} \equalindistribution \eptc$, and $\htc^{e*} \equalindistribution \ehtc$, where all are mutually independent, the approximation for the ruin probability \approxruinprobabilitymix in \Cref{Approximation} then takes the form
\begin{align*}
    \approxruinprobabilitymix =&\  \approxgerbermixruinq[0] = \ruinprobability +
    \epsilon \narrate \Bigg(
              \int_{[0,\infty)}  \densityresolventmatrix
\Big( \mean[h] \htecd[dz] -\mean[p] \ptecd[dz] \Big)
    +  \int_{[0,\infty)}  \correctionresolventq
 \mean[p] \ptecd[dz]  \Bigg) \\
    =& \ \ruinprobability + \frac{\epsilon \narrate (\mean[h] - \mean[p])}{c - \narrate \mean[p]}    \pr (M \leq u)
    - \frac{\epsilon \narrate}{c - \narrate \mean[p]}  \big( \mean[h] \pr (M + \ehtc \leq u) - \mean[p] \pr (M + \eptc \leq u) \big) \\
    &+ \frac{\epsilon \narrate \mean[p]}{(c - \narrate \mean[p])^2} \big( \narrate \mean[h] \pr (M +M^* + \ehtc \leq u)  - \narrate \mean[p] \pr (M +M^* + \eptc \leq u)  \big) \\
    &- \frac{\epsilon \narrate \mean[p]}{(c - \narrate \mean[p])^2} \big( \narrate \mean[h] \pr (M +M^* + \ehtc + \eptc \leq u)  - \narrate \mean[p] \pr (M +M^* + \eptc + \ptc^{e*} \leq u)  \big) = \dots \\
    =& \ \ruinprobability + \frac{\epsilon \narrate \mean[h]}{c - \narrate \mean[p]}  \big( \pr (M +M^* + \ehtc > u) - \pr (M >u) \big) \\
    & -  \frac{\epsilon \narrate \mean[p]}{c - \narrate \mean[p]}\big(  \big( \pr (M +M^* + \eptc > u) - \pr (M >u) \big).
\end{align*}
Notice that we have used the identity $\pr (M \in dx)* \ptecd[dx] = \pr (M + \eptc \in  dx) = \displaystyle \frac{c}{\narrate \mean[p]}\pr (M \in dx) - \frac{c - \narrate \mean[p] }{\narrate \mean[p]}$ to simplify the convoluted probabilities $\pr (M +M^* + \ehtc + \eptc \leq u)$ and $\pr (M +M^* + \eptc + \ptc^{e*} \leq u)$, while $\int_0^\infty \htecd[dz] =\int_0^\infty \ptecd[dz] = 1$, because \htecds and \ptecds are distributions. The above result is in accord with \cite{vatamidou2013correctedrisk}, which further accommodates the whole series expansion of \approxruinprobabilitymix.

\subsection{Heavy-tailed asymptotics}
Some of the summands in the $\epsilon$ term of our corrected phase\-/type approximation \approxgerbermixq are convolutions of a single excess heavy\-/tailed claim size \ehtc with some light tailed distributions. We make rigorous in this section that these summands administer a heavy\-/tailed behavior to the approximation. Note that we don't aim at deriving exact tail asymptotics. We rather intend to provide a feeling for the asymptotic behavior of the approximation and thus treat this section more as a discussion than a strict proof.

For our purposes, we assume that \htcd belongs to the class of long\-/tailed distributions, that is, for any fixed $y>0$ we have $\lim_{x \rightarrow \infty} \chtcd(x+y)/\chtcd(x)=1$. In addition, recall that \penaltyfunction is a bounded function, i.e.\ $\penaltyfunction \leq a$ for some constant $a>0$. Finally, we only focus on the case $\forceofinterest>0$, since the proof for $\forceofinterest=0$ can be handled as a limiting case.

Since \penaltyfunction is bounded, it is clear from \Cref{Eq.Gerber-Shiu definition} that the GS function is bounded from above by the ruin probability multiplied by a constant. In this case, from \cite[Lemma~1 \& Equation~(19)]{asmussen2004russian},
\begin{equation*}
    \gerberq = \constantfactor \narrate
             \int_{[0,\infty)} \int_{[0,\infty)} \penaltyf{y}{z}  \densityresolventmatrixq_{(1,1)} \ptcd(z + dy) dz  \quad \leq \quad \beta e^{-\gamma u},
\end{equation*}
for some constant $\beta$ and a strictly positive number $\gamma$. We will show that each of the remaining two terms of \Cref{Approximation} can be split in a part which is also bounded by an exponential and a heavy\-/tailed part that is responsible for the tail behavior of the approximation \approxgerbermixq.

\paragraph{Term 1.} We start our discussion with the term
\begin{equation*}\label{Eq. Term 1}
     \epsilon \constantfactor \narrate \int_{[0,\infty)} \int_{[0,\infty)} \penaltyf{y}{z}  \correctionresolventq \ptcd (z + dy) dz,
\end{equation*}
where \correctionresolventq is provided in \Cref{Corollary. Resolvent ID perturbed}. Clearly, we can write $\correctionresolventq = \ptcorrectionresolventq + \htcorrectionresolventq$, where
\begin{align}
        \ptcorrectionresolventq
    =& - \narrate \mean[p] \ptecd[u] * \big( \scalematrix[du] \big)_{(1,1)} *\big( \scalematrix[du] \big)_{1\bullet} \cdot \Big( e^{-\leftsolutionq z} \Big)_{\bullet 1} \notag \nonumber\\
    &+ \big( \scalematrix[u] \big)_{1\bullet} \cdot \Big( \big( \leftkernelq \big)^{-1} e^{- \diagonalmatrixq z} \leftkernelperturbationq[1] - z \big( \leftkernelq \big)^{-1} e^{- \diagonalmatrixq z} \diagonalmatrixperturbationq[1] \leftkernelq - \big( \leftkernelq \big)^{-1} \leftkernelperturbationq[1] e^{-\leftsolutionq z} \Big)_{\bullet 1} \notag \nonumber \\
    &+ \narrate \int_0^{u-z}\mean[p] \ptecd[u-z-x] \big( \scalematrix[dx] \big)_{(1,1)}^{*2}, \label{Eq. PH part of the corrected resolvent}
\intertext{and}
  \htcorrectionresolventq
    =& \narrate  \mean[h] \htecd[u] * \big( \scalematrix[du] \big)_{(1,1)} *\big( \scalematrix[du] \big)_{1\bullet} \cdot \Big( e^{-\leftsolutionq z} \Big)_{\bullet 1} \notag \\
    &- \narrate \int_0^{u-z} \mean[h] \htecd[u-z-x]  \big( \scalematrix[dx] \big)_{(1,1)}^{*2}. \label{Eq. HT part of the corrected resolvent}
\end{align}

\begin{lemma}\label{Lemma: Exponentially bounded}
The term
\begin{equation*}
    \qquad \epsilon \constantfactor \narrate \int_{[0,\infty)} \int_{[0,\infty)} \penaltyf{y}{z}  \ptcorrectionresolventq \ptcd (z + dy) dz
\end{equation*}
is bounded by an exponential.
\end{lemma}
\begin{proof}
Indeed, from \cite[Equation~(2.3)]{czarna2018fluctuation}, \cite[Theorem~1]{ivanovs2012occupation}, and \cite[Corollary~1]{ivanovs2014potential}, we have that
\begin{equation}\label{Eq. Scale matrix exponential limit}
\lim_{u\to+\infty} e^{-\diagonalmatrixq u}\scalematrix[u] = \lim_{u\to+\infty} \scalematrix[u] e^{-\leftsolutionq u} = \leftkernelq.
\end{equation}
Now using the fact that $\omega$ is bounded, we can write
\begin{align*}
    \epsilon \constantfactor \narrate &\int_0^\infty \int_0^\infty \penaltyf{y}{z}  \ptcorrectionresolventq \ptcd(z+ dy) dz
    \ \leq\
    \epsilon \constantfactor \narrate a \int_0^\infty \ptcorrectionresolventq \cptcd(z ) dz \\
    &= \epsilon \constantfactor \narrate a \int_0^u \ptcorrectionresolventq \cptcd(z ) dz
    + \epsilon \constantfactor \narrate a \int_u^\infty \ptcorrectionresolventq \cptcd(z ) dz \\
    &\leq \epsilon \constantfactor \narrate a \underbrace{\int_0^u \ptcorrectionresolventq \cptcd(z ) dz}_{:=I_1}
    + \epsilon \constantfactor \narrate a \underbrace{ \cptcd(u) \int_u^\infty \ptcorrectionresolventq  dz}_{:=I_2}.
\end{align*}
It is well known that any phase\-/type distribution is light\-/tailed, that is there exist $\gamma_p$ and $\beta_p$ such that
\begin{equation}\label{Eq. Exponential tail phase-type}
    \lim_{u\to+\infty} \cptcd(u) e^{\gamma_p u}=\beta_p>0.
\end{equation}

Note that the integral $\int_u^\infty e^{-\leftsolutionq z}dz$ is of order $e^{-\leftsolutionq u}$. Consequently, it follows that $\int_u^\infty   \ptcorrectionresolventq  dz$ tends to a constant as $u\to+\infty$
and hence the term $I_2$ is dominated by some exponential function due to \eqref{Eq. Exponential tail phase-type}.

For $I_1$, observe that each summand of $\ptcorrectionresolventq$ behaves like
\scalematrix[u] for fixed $z\leq u$. In other words, there exists by \eqref{Eq. Scale matrix exponential limit} a $\gamma_W>0$ such that
$\lim_{u\to +\infty} \scalematrix[u]e^{\gamma_W u}=\leftkernelq_0$ for a non\-/zero and finite matrix $\leftkernelq_0$. As a result, all the increments in \eqref{Eq. PH part of the corrected resolvent} multiplied by $e^{\gamma_W u}$ tend to some constants, the sum of which must be equal to zero because the Gerber\-/Shiu function that dominates the sum goes to zero.
The asymptotics of $I_1$ are then determined by the sum of the speeds that each increment tends to a constant. Remark that the speeds of convergence of
$\displaystyle - \narrate \mean[p]\int_0^ u \ptecd[u] * \big( \scalematrix[du] \big)_{(1,1)} *\big( \scalematrix[du] \big)_{1\bullet} \cdot \Big( e^{-\leftsolutionq z} \Big)_{\bullet 1}dz
$ and $\displaystyle \narrate \int_0^ u \int_0^{u-z}\mean[p] \ptecd[u-z-x] \big( \scalematrix[dx] \big)_{(1,1)}^{*2}dz$ are of order $\cptcd(u)$ and therefore exponential by \eqref{Eq. Exponential tail phase-type}.

To analyze the second increment in \eqref{Eq. PH part of the corrected resolvent} and complete the proof, it suffices to prove that the speed of convergence in \eqref{Eq. Scale matrix exponential limit} is exponential.
We focus on the first convergence, while the second can be analyzed in a similar way. We have from \cite[Theorem~1]{ivanovs2012occupation} that
 $e^{-\diagonalmatrixq u}\scalematrix[u] =\leftkernelq(u)$, where $\leftkernelq(u)$ is the matrix of expected occupation times at $0$ up to the first passage over $u$ of the process \levynew exponentially killed with intensity $\forceofinterest >0$, i.e.\ up to $\tau_u^+=\inf\{t\geq 0: \levynew \geq u\}$, where \levynew starts right now at level $0$. From the comment below this aforementioned theorem, it follows that
$\lim_{u\to\infty} \leftkernelq(u) =\leftkernelq$, where \leftkernelq is the matrix of expected occupation times
at $0$. If $\tau_u=\inf\{t\geq 0: \levynew =-u\}$ is the first hitting time of level $-u$, it then follows from \cite[Equation~(10)]{ivanovs2012occupation} that
\begin{equation*}
    \leftkernelq-\leftkernelq(u)=\pr_0 \big( J(\tau_u^+) \big)\pr_0\big( J(\tau_u), \tau<+\infty \big)\leftkernelq.
\end{equation*}
To prove that the speed is exponential, it suffices then to show that
$\pr_0 \big( J(\tau_u) \big)$ has an exponential bound. Consequently, it holds that $\pr\big( J(\tau_u) \big)=\int_0^\infty \pr_u \big( | \levynew[\tau] | \in dy, J(\tau), \tau<+\infty \big) \pr_0 \big( J(\tau_y^+) \big)dy$ from \cite[Proposition~7]{ivanovs2012occupation} and the Markov property. Thus, it only remains to show that $\ruinprobability=\pr_u \big( J(\tau), \tau<+\infty \big)$ has an exponential bound, which is true due to \cite[Equation~(19)]{asmussen2004russian}.
\end{proof}

\noindent For the term
\begin{equation*}
    \epsilon \constantfactor \narrate \int_{[0,\infty)} \int_{[0,\infty)} \penaltyf{y}{z}  \htcorrectionresolventq \ptcd (z + dy) dz,
\end{equation*}
we first analyze further \htcorrectionresolventq as follows
\begin{align*}
    \htcorrectionresolventq
    =&\ \narrate  \mean[h] \int_0^u \htecd[u-x] \big( \scalematrix[dx] \big)_{(1,1)} *\big( \scalematrix[dx] \big)_{1\bullet} \cdot \Big( e^{-\leftsolutionq z} \Big)_{\bullet 1} \\
    &- \narrate \int_0^{u-z} \mean[h] \htecd[u-z-x]  \big( \scalematrix[dx] \big)_{(1,1)}^{*2} \\
    =&\ \narrate  \mean[h] \int_0^u \big( \scalematrix[dx] \big)_{(1,1)} *\big( \scalematrix[dx] \big)_{1\bullet} \cdot \Big( e^{-\leftsolutionq z} \Big)_{\bullet 1} - \narrate \mean[h]  \big( \scalematrix[u-z] \big)_{(1,1)}^{*(2)}\\
    &- \narrate  \mean[h] \int_0^u \htecdcom[u-x] \big( \scalematrix[dx] \big)_{(1,1)} *\big( \scalematrix[dx] \big)_{1\bullet} \cdot \Big( e^{-\leftsolutionq z} \Big)_{\bullet 1} \\
    &+ \narrate \mean[h] \int_0^{u-z}  \htecdcom[u-z-x]  \big( \scalematrix[dx] \big)_{(1,1)}^{*2},
\end{align*}
where we used the notation $\big(\scalematrix[x] \big)_{(1,1)}^{*(2)} = \big(\scalematrix[dx] \big)_{(1,1)}*\big(\scalematrix[x] \big)_{(1,1)}$. Thus,
\begin{align}
    \epsilon \constantfactor &\narrate \int_0^\infty \int_0^\infty  \penaltyf{y}{z}  \htcorrectionresolventq \ptcd (z + dy) dz \notag \\
    \leq& \epsilon a \constantfactor \narrate \mean[p] \narrate \mean[h]  \int_0^\infty \int_0^u \big( \scalematrix[dx] \big)_{(1,1)} *\big( \scalematrix[dx] \big)_{1\bullet} \cdot \Big( e^{-\leftsolutionq z} \Big)_{\bullet 1} \ptecd[dz] \notag \\
    &+ \epsilon a \constantfactor \narrate \mean[p] \narrate \mean[h] \int_0^u \big( \scalematrix[u-z] \big)_{(1,1)}^{*(2)} \ptecd[dz] \notag \\
    &-\epsilon \constantfactor \narrate \int_0^\infty \hspace{-0.3cm} \int_0^\infty \hspace{-0.2cm} \penaltyf{y}{z}  \narrate  \mean[h] \int_0^u \hspace{-0.2cm} \htecdcom[u-x] \big( \scalematrix[dx] \big)_{(1,1)} *\big( \scalematrix[dx] \big)_{1\bullet} \cdot \Big( e^{-\leftsolutionq z} \Big)_{\bullet 1} \ptcd(z + dy)   dz \notag \\
    &+ \epsilon \constantfactor \narrate \int_0^\infty \int_0^\infty \penaltyf{y}{z}  \bigg( \narrate \mean[h] \int_0^{u-z}  \htecdcom[u-z-x]  \big( \scalematrix[dx] \big)_{(1,1)}^{*2} \bigg)
 \ptcd(z + dy)   dz \notag \\
 =& \epsilon a \constantfactor \narrate \mean[p] \narrate \mean[h] \big( \scalematrix[du] \big)_{(1,1)} *\big( \scalematrix[u] \big)_{1\bullet} \cdot \bigg( \big( \leftkernelq \big)^{-1} \diag \Big( e^{\ltptec[{\simplerootq[i]}]} \Big)_{i \in \statespace} \leftkernelq \bigg)_{\bullet 1} \notag \\
 &+ \epsilon a \constantfactor \narrate \mean[p] \narrate \mean[h] \ptecd[u] *\big( \scalematrix[du] \big)_{(1,1)}^{*2}
 + \epsilon \constantfactor \narrate \int_0^u \htecdcom[u-x] \varphi(x) dx, \label{Eq. Heavy part of Term 1}
 \intertext{with}
    \varphi(x)
    =& -\narrate  \mean[h] \big( \scalematrix[dx] \big)_{(1,1)} *\big( \scalematrix[dx] \big)_{1\bullet} \cdot \int_0^\infty \int_0^\infty \penaltyf{y}{z}  \Big( e^{-\leftsolutionq z} \Big)_{\bullet 1} \ptcd(z + dy) dz \notag \\
    &+ \narrate  \mean[h]  \big( \scalematrix[dx] \big)_{(1,1)}^{*2} \int_0^\infty \int_0^x \penaltyf{y}{z} \ptcd(z + dy) dz.\notag
\end{align}
Using similar arguments as in \Cref{Lemma: Exponentially bounded}, we can prove that the first two terms in \eqref{Eq. Heavy part of Term 1} are bounded by an exponential.

\paragraph{Term 2.} We continue now our analysis with the term
\begin{equation*}
    \epsilon \constantfactor \narrate
        \int_{[0,\infty)} \int_{[0,\infty)} \penaltyf{y}{z} \densityresolventmatrixq_{(1,1)}
 \htcd(z + dy) dz.
\end{equation*}
Using $e^{-\leftsolutionq z} = \big( \leftkernelq \big)^{-1} e^{-\diagonalmatrixq z} \leftkernelq$, we then write
\begin{align}
    \ \epsilon \constantfactor \narrate &\int_0^\infty \int_0^\infty \penaltyf{y}{z}  \densityresolventmatrixq_{(1,1)}
 \htcd(z + dy) dz
   \leq \ \epsilon \constantfactor \narrate a \mean[h] \int_0^\infty \densityresolventmatrixq_{(1,1)} \htecd[dz] \notag \\
   = &\ \epsilon \constantfactor \narrate a \mean[h] \big( \scalematrix[u] \big)_{1 \bullet} \int_0^\infty \big( e^{-\leftsolutionq z} \big)_{\bullet 1} \htecd[dz]
   - \epsilon \constantfactor \narrate a \mean[h] \int_0^u \big( \scalematrix[u-z] \big)_{(1,1)} \htecd[dz] \notag \\
  =&\ \epsilon \constantfactor \narrate a \mean[h] \big( \scalematrix[u] \big)_{1 \bullet} \cdot \Bigg( \big( \leftkernelq \big)^{-1} \diag \Big( e^{\lthtec[{\simplerootq[i]}]} \Big)_{i \in \statespace} \leftkernelq \Bigg)_{\bullet 1} \notag \\
  &\
   - \epsilon \constantfactor \narrate a \mean[h]  \htecdcom[u-z] \big( \scalematrix[z] \big)_{(1,1)} \Bigm \vert_{0}^u
   + \epsilon \constantfactor \narrate a \mean[h] \int_0^u \htecdcom[u-z] \big( \scalematrix[dz] \big)_{(1,1)} \notag \\
  =&\ \epsilon \constantfactor \narrate a \mean[h] \big( \scalematrix[u] \big)_{1 \bullet} \cdot \Bigg( \big( \leftkernelq \big)^{-1} \diag \Big( e^{\lthtec[{\simplerootq[i]}]} \Big)_{i \in \statespace} \leftkernelq \Bigg)_{\bullet 1}
  - \epsilon \constantfactor \narrate a \mean[h]  \big( \scalematrix[u] \big)_{(1,1)} \notag \\
  &+ \epsilon \constantfactor \narrate a \mean[h] \big( \scalematrix[0]  \big)_{(1,1)} \htecdcom[u] + \epsilon \constantfactor \narrate a \mean[h] \int_0^u \htecdcom[u-z] \big( \scalematrix[dz] \big)_{(1,1)}.\label{Eq. Heavy part in Term 2}
\end{align}
Again, we can prove that the first two terms in \eqref{Eq. Heavy part in Term 2} are bounded by an exponential by using similar arguments as in \Cref{Lemma: Exponentially bounded}.

To show now the main result of this section, we need additional lemma.
We need also the definition of $h$-insensitivity from \cite{foss-IHTSD}, which we call here $\phi$\-/insensitivity to avoid confusion with the index $h$ in \htecds. In particular, given a non\-/decreasing function $\phi$ on \reals[+] such that $\phi(x) \rightarrow \infty$ as $x \rightarrow \infty$, a long\-/tailed distribution $F$ is called $\phi$\-/insensitive (or $\phi$\-/flat) if $\bar{F}\big( x \pm \phi(x) \big) \sim \bar{F}(x)$ as $x \rightarrow \infty$.

\begin{lemma}\label{Lemma: limit for integral}
For any integrable function $\zeta(y)$
we assume that
\begin{equation}\label{lightzeta}
\lim_{u\rightarrow\infty}\frac{\int_{\phi(u)}^\infty\zeta(y)dy}{\htecdcom[u]}=0,
\end{equation}
where \htecds is $\phi$\-/insensitive.
Then it holds that
  \begin{equation*}
      \lim_{u \rightarrow \infty} \frac{\int_0^u \htecdcom[u-y] \zeta(y)dy}{\htecdcom[u]} = \int_0^\infty \zeta(y) dy < \infty.
  \end{equation*}
\end{lemma}
\begin{proof}
We know from \cite[Theorem~2.28, p.23]{foss-IHTSD} that \htecd is also long\-/tailed and we can then write for any function $\phi$ such that \htecds is $\phi$\-/insensitive,
\begin{align}
    \frac{\int_0^u \htecdcom[u-y] \zeta(y)dy}{\htecdcom[u]}
    &=
    \int_0^{\phi(u)} \frac{ \htecdcom[u-y] }{\htecdcom[u]}\zeta(y)dy
    +
    \int_{\phi(u)}^u \frac{ \htecdcom[u-y] }{\htecdcom[u]} \zeta(y)dy \notag \\
    &\leq  \int_0^{\phi(u)} \frac{ \htecdcom[u-\phi(u)] }{\htecdcom[u]}\zeta(y)dy
    +
    \int_{\phi(u)}^u \frac{\htecdcom[u-y] }{\htecdcom[u]}\zeta(y)dy. \label{Eq. Inequalities for tail behavior}
\end{align}
For fixed $y$, it holds by the definition of long\-/tailed distributions that $\lim_{u \rightarrow \infty} \htecdcom[u-y] /\htecdcom[u] =1$. Moreover, $\lim_{u \rightarrow \infty} \htecdcom[u-\phi(u)] /\htecdcom[u] = 1$ by virtue of $\phi$\-/insensitivity. Consequently, for $u$ large enough, there exist small $\epsilon_1,\epsilon_2>0$ such that $\htecdcom[u-\phi(u)] /\htecdcom[u] \leq 1+\epsilon_1$. Moreover, $\displaystyle \int_{\phi(u)}^u \frac{\htecdcom[u-y] }{\htecdcom[u]}\zeta(y)dy\leq
\int_{\phi(u)}^\infty \zeta(y)dy /\htecdcom[u]$, which tends to zero by \eqref{lightzeta}, hence can be bounded by some $\epsilon_2$ for sufficiently large $u$.
Thus, we can write using \Cref{Eq. Inequalities for tail behavior} that
\begin{align*}
    \lim_{u \rightarrow \infty}\frac{\int_0^u \htecdcom[u-y] \zeta(y)dy}{\htecdcom[u]}
    &\leq  (1+\epsilon_1)\int_0^{\phi(u)} \zeta(y)dy
    +\epsilon_2
        \leq (1+\epsilon_1) \int_0^\infty \zeta(y)dy +\epsilon_2,
\end{align*}
where we used that $\zeta(y)$ is integrable. The required result now follows by letting $\max\{\epsilon_1, \epsilon_2\} \rightarrow 0$.
\end{proof}

\begin{theorem}\label{Theorem: Asymptotics}
When \penaltyfunction is a bounded function with $\penaltyfunction \leq a$, we have
\begin{equation*}
    \lim_{u \rightarrow \infty} \frac{\approxgerbermixq}{\htecdcom[u]} \leq  \epsilon \constantfactor \narrate \Bigg(  a \mean[h] \big( \scalematrix[0]  \big)_{(1,1)} + \int_0^\infty \kappa(x)dx \Bigg),
\end{equation*}
where
\begin{align*}
    \kappa(x)
    =& -\narrate  \mean[h] \big( \scalematrix[dx] \big)_{(1,1)} *\big( \scalematrix[dx] \big)_{1\bullet} \cdot \int_0^\infty \int_0^\infty \penaltyf{y}{z}  \Big( e^{-\leftsolutionq z} \Big)_{\bullet 1} \ptcd(z + dy) dz \\
    &+ \narrate  \mean[h]  \big( \scalematrix[dx] \big)_{(1,1)}^{*2} \int_0^\infty \int_0^x \penaltyf{y}{z} \ptcd(z + dy) dz
    + a \mean[h]  \big( \scalematrix[dx] \big)_{(1,1)}.
\end{align*}
The equality holds exactly when $\penaltyfunction \equiv a$.
\end{theorem}
\begin{proof}
The terms of \approxgerbermixq that are bounded by an exponential vanish in the tail and the asymptotic behavior of the approximation can only be attributed to the terms involving \htecds. For the latter terms, it suffices to note that the function $\kappa(x)$ is integrable and satisfies
\eqref{lightzeta} due to \eqref{Eq. Scale matrix exponential limit} (since $\htecdcom$ is heavy-tailed) and apply \Cref{Lemma: limit for integral}. The integrability of $\kappa(x)$ comes from the integrability of the scale matrix \scalematrix[dx], which can be proven by using similar arguments as in \Cref{Lemma: Exponentially bounded}.
\end{proof}

\section{Conclusions}\label{Section: Conslusions}
This work complements other investigations of the Gerber\-/Shiu function for risk processes with two-sided jumps; see e.g.\ \cite{albrecher2010direct,ji2010gerber,labbe2009expected}. However, they all assume special distributions with light tails for the claim sizes. In contrast, our paper allows for additional heavy\-/tailed claims that may appear with a small probability.

Specifically, we combined perturbation analysis with fluid embedding to construct approximations for the Gerber\-/Shiu function. The developed approximations have a proven $O(\epsilon^2)$ error and a heavy\-/tailed behavior in the tail. Moreover, the derived closed\-/form formulas are not only suitable to produce numerical estimates for the GS function but allow us additionally to study theoretical properties of the approximation. Moreover, it could be feasible to calculate exact tail asymptotics when all involved quantities are defined explicitly.

Note that one possible generalization is related with the addition of an independent Brownian component to the risk process. However, the analysis becomes much more complex since Brownian motion could cause ruin by creeping and this case is not treated by our approach. Finally, another possible extension that needs to be researched by its own is to construct approximations for the GS function of a general MAP with all upward jumps of phase\-/type.

\section*{Acknowledgements}
The work of Zbigniew Palmowski is partially supported by National Science Centre Grant No. 2016/23/B/HS4/00566 (2017-2020). Eleni Vatamidou acknowledges financial support from the Swiss National Science Foundation Project 200021\_168993.

\phantomsection
\addcontentsline{toc}{section}{References}
\printbibliography

\end{document}

%% file: macros.tex
\newcommand{\reals}[1][]{\ensuremath{\mathds{R}^{#1}}\xspace}        

\newcommand{\fun}[2][s]{\ensuremath{{#2}(#1)}}     

\newcommand{\ptc}[1][]{\ensuremath{B_{#1}}\xspace}                    
\newcommand{\eptc}[1][]{\ensuremath{B^e_{#1}}\xspace}                 
\newcommand{\htc}[1][]{\ensuremath{C_{#1}}\xspace}                    
\newcommand{\ehtc}[1][]{\ensuremath{C^e_{#1}}\xspace}                 
\newcommand{\gpc}[1][]{\ensuremath{U^{(+)}_{#1}}\xspace}              
\newcommand{\gnc}[1][]{\ensuremath{U_{#1}}\xspace}              
\newcommand{\mixc}[1][]{\ensuremath{U^{(-)}_{\epsilon #1}}\xspace}    

\newcommand{\parrate}{\ensuremath{\lambda^{(+)}}\xspace}          
\newcommand{\narrate}{\ensuremath{\lambda^{(-)}}\xspace}          

\newcommand{\pphases}{\ensuremath{m^{\footnotesize{(+)}}}\xspace}          

\newcommand{\pintm}{\ensuremath{\mathbf{T}^{(+)}}\xspace}          

\newcommand{\pinp}{\ensuremath{\boldsymbol \alpha^{(+)}}\xspace}      

\newcommand{\pexr}{\ensuremath{\mathbf{t}^{(+)}}\xspace}      

\newcommand{\levymix}[1][t]{\ensuremath{\fun[#1]{X_\epsilon}}\xspace}       
\newcommand{\levymixcase}[2][t]{\ensuremath{\fun[#1]{X_{\epsilon,#2}}}\xspace}       
\newcommand{\levymixnew}[1][t]{\ensuremath{\fun[#1]{Y_\epsilon}}\xspace}       
\newcommand{\levynew}[1][t]{\ensuremath{\fun[#1]{Y}}\xspace}       
\newcommand{\jumpprocess}[1][t]{\ensuremath{\fun[#1]{J_\epsilon}}\xspace}       
\newcommand{\ppp}{\ensuremath{N^{(+)}}\xspace}       
\newcommand{\npp}{\ensuremath{N}\xspace}       
\newcommand{\nppmix}{\ensuremath{N^{(-)}_\epsilon}\xspace}       
\newcommand{\ruintimemix}{\ensuremath{\tau_{\epsilon}}\xspace}       

\newcommand{\emeasure}[1][u]{\ensuremath{\mathds{E}_{#1}}\xspace}    

\newcommand{\deficitruinmix}{\ensuremath{-\levymix[\ruintimemix]}\xspace}       
\newcommand{\surpluspriorruinmix}{\ensuremath{\levymix[\ruintimemix-]}\xspace}       
\newcommand{\gerbermixq}[1][q]{\ensuremath{GS_{\penaltyfunction}^{\epsilon}\left(u,#1\right)}\xspace}       
\newcommand{\approxgerbermixq}[1][q]{\ensuremath{\widehat{GS_{\penaltyfunction}^{\epsilon}}\left(u,#1\right)}\xspace}       

\newcommand{\levy}[1][t]{\ensuremath{\fun[#1]{X}}\xspace}       
\newcommand{\ruintime}{\ensuremath{\tau}\xspace}       
\newcommand{\deficitruin}{\ensuremath{-\levy[\ruintime]}\xspace}       
\newcommand{\surpluspriorruin}{\ensuremath{\levy[\ruintime-]}\xspace}       
\newcommand{\penaltyf}[2]{\ensuremath{\omega\left(#1,#2\right)}\xspace}    
\newcommand{\penaltyfunction}{\ensuremath{\omega}\xspace}    
\newcommand{\penaltyfunctionruin}{\ensuremath{1}\xspace}    
\newcommand{\gerberq}[1][q]{\ensuremath{GS_{\penaltyfunction}\left(u,#1\right)}\xspace}       
\newcommand{\approxgerbermixruinq}[1][q]{\ensuremath{\widehat{GS_{\penaltyfunctionruin}^{\epsilon}}\left(u,#1\right)}\xspace}       
\newcommand{\densityresolventq}[1][q]{\ensuremath{r^{(#1)}(u,z)}\xspace}       
\newcommand{\densityresolventmatrixq}[1][q]{\ensuremath{\mathbf{r}^{(#1)}(u,z)}\xspace}       
\newcommand{\densityresolventmatrix}{\ensuremath{\mathbf{r}(u,z)}\xspace}       
\newcommand{\densityresolventmatrixmixq}[1][q]{\ensuremath{\mathbf{r}^{(#1)}_\epsilon(u,z)}\xspace}       
\newcommand{\correctionresolventq}[1][q]{\ensuremath{v^{(#1)}(u,z)}\xspace}       
\newcommand{\ptcorrectionresolventq}[1][q]{\ensuremath{v^{(#1)}_p(u,z)}\xspace}       
\newcommand{\htcorrectionresolventq}[1][q]{\ensuremath{v^{(#1)}_h(u,z)}\xspace}       
\newcommand{\ruinprobability}[1][u]{\ensuremath{\psi(#1)}\xspace}       
\newcommand{\approxruinprobabilitymix}[1][u]{\ensuremath{\widehat{\psi}_{\epsilon}(#1)}\xspace} 

\newcommand{\statespace}{\ensuremath{E}\xspace}     

\newcommand{\lapexpmix}[1][s]{\ensuremath{\phi_\epsilon(#1)}\xspace}      
\newcommand{\lapexp}[1][s]{\ensuremath{\phi(#1)}\xspace}      
\newcommand{\mfmixq}[1][s]{\ensuremath{\mathbf{F}^{(q)}_\epsilon(#1)}\xspace}      
\newcommand{\mf}[1][s]{\ensuremath{\mathbf{F}(#1)}\xspace}      
\newcommand{\mfq}[1][s]{\ensuremath{\mathbf{F}^{(q)}(#1)}\xspace}      
\newcommand{\mfmix}[1][s]{\ensuremath{\mathbf{F}_\epsilon(#1)}\xspace}      

\newcommand{\adjointmfmixq}[1][s]{\ensuremath{\mathbf{\mathfrak{F}}^{(q)}_\epsilon(#1)}\xspace}      
\newcommand{\adjointmfq}[1][s]{\ensuremath{\mathbf{\mathfrak{F}}^{(q)}(#1)}\xspace}      
\newcommand{\adjointmfelementq}[1][ij]{\ensuremath{\mathfrak{f}^{(q)}_{#1}}\xspace}      

\newcommand{\adjointmxk}[1][]{\ensuremath{\ltm{\mathbf{\mathfrak{K}}_{#1}}}\xspace}     

\newcommand{\e}[1][]{\ensuremath{\mathds{E}{#1}}\xspace}    
\newcommand{\pr}[1][]{\ensuremath{\mathds{P}{#1}}\xspace}   
\newcommand{\adj}[1]{\ensuremath{\textup{adj}#1}}   
\newcommand{\vect}[1]{\ensuremath{\mathbf{#1}}\xspace}  
\newcommand{\diag}{\ensuremath{\text{diag}}\xspace}  
\newcommand{\n}[1][]{\ensuremath{\mathcal{N}}\xspace}    
\newcommand{\mean}[1][]{\ensuremath{\mu_{#1}}\xspace}    
\newcommand{\indfun}[1][]{\ensuremath{\mathds{1}_{#1}}\xspace}       
\newcommand{\equalindistribution}{\ensuremath{\stackrel{\mathfrak{D}}{=}}\xspace}					

\newcommand{\limitmatrixmix}[1][]{\ensuremath{\mathbf{A}_{\epsilon}}\xspace}     
\newcommand{\limitmatrixdis}[1][]{\ensuremath{\mathbf{A}^{\bullet}_{\epsilon}}\xspace}     

\newcommand{\matrixeigenmix}[1][]{\ensuremath{\mathbf{L}_{\epsilon}}\xspace}     
\newcommand{\matrixeigendis}[1][]{\ensuremath{\mathbf{L}^{\bullet}_{\epsilon}}\xspace}     

\newcommand{\zeromatrix}{\ensuremath{\mathbb{O}}\xspace}     

\newcommand{\unitvector}[1][]{\ensuremath{\mathbf{e}_{#1}}\xspace}           
\newcommand{\unitcolumn}{\ensuremath{\mathbf{1}}\xspace}           

\newcommand{\unitmatrix}{\ensuremath{\mathbf{U}}\xspace}           

\newcommand{\im}[1][]{\ensuremath{\mathbf{I}_{#1}}\xspace}           

\newcommand{\mxk}[1][]{\ensuremath{\ltm{\mathbf{K}_{#1}}}\xspace}     
\newcommand{\mxkelement}[2][s]{\ensuremath{k^{#2}(#1)}\xspace}     

\newcommand{\leftsolutionmixq}{\ensuremath{\mathbf{R}^{(q)}_\epsilon}\xspace}           
\newcommand{\leftsolutionq}{\ensuremath{\mathbf{R}^{(q)}}\xspace}           
\newcommand{\diagonalmatrixq}{\ensuremath{\mathbf{\Lambda}^{(q)}}\xspace}           
\newcommand{\diagonalmatrixmixq}{\ensuremath{\mathbf{\Lambda}^{(q)}_\epsilon}\xspace}           
\newcommand{\diagonalmatrixperturbationq}[1][]{\ensuremath{\mathbf{\Delta}^{(q)}_{#1}}\xspace}           
\newcommand{\leftkernelq}{\ensuremath{\mathbf{L}^{(q)}}\xspace}           
\newcommand{\leftkernelperturbationq}[1][]{\ensuremath{\mathbf{H}^{(q)}_{#1}}\xspace}           
\newcommand{\leftkernelmixq}{\ensuremath{\mathbf{L}^{(q)}_\epsilon}\xspace}           

\newcommand{\scalematrixmix}[1][x]{\ensuremath{\mathbf{W}^{(q)}_{\epsilon}(#1)}\xspace}     
\newcommand{\scalematrix}[1][x]{\ensuremath{\mathbf{W}^{(q)}(#1)}\xspace}     
\newcommand{\scalematrixsimple}[1][x]{\ensuremath{\mathbf{W}(#1)}\xspace}     

\newcommand{\scalefunctionmix}{\ensuremath{\mathbf{W}^{(q)}_{\epsilon}}\xspace}     

\newcommand{\markovprocess}[1][t]{\ensuremath{J(t)}\xspace}     
\newcommand{\forceofinterest}{\ensuremath{q}\xspace}        
\newcommand{\forceofinterestmatrix}{\ensuremath{\bm{\mathcal{Q}}}\xspace}        

\newcommand{\ltm}[2][s]{\ensuremath{#2(#1)}\xspace}             
\newcommand{\ltptec}[1][s]{\ensuremath{\tilde{F}^e_{p}(#1)}\xspace}        
\newcommand{\lthtec}[1][s]{\ensuremath{\tilde{F}^e_{h}(#1)}\xspace}        

\newcommand{\ptcd}{\ensuremath{F_{p}}\xspace}        
\newcommand{\cptcd}{\ensuremath{\bar{F}_{p}}\xspace}        
\newcommand{\htcd}{\ensuremath{F_{h}}\xspace}        
\newcommand{\chtcd}{\ensuremath{\bar{F}_{h}}\xspace}        
\newcommand{\pptcd}{\ensuremath{F^{(+)}}\xspace}        
\newcommand{\mixcd}[1][\epsilon]{\ensuremath{F_{#1}^{(-)}}\xspace}        
\newcommand{\ptecd}[1][x]{\ensuremath{F^e_{p}(#1)}\xspace}        
\newcommand{\ptecds}{\ensuremath{F^e_{p}}\xspace}        
\newcommand{\ptecconv}[2][x]{\ensuremath{\big(F^e_{p}\big)^{*#2}(#1)}\xspace}        
\newcommand{\htecd}[1][x]{\ensuremath{F^e_{h}(#1)}\xspace}        
\newcommand{\htecds}{\ensuremath{F^e_{h}}\xspace}        
\newcommand{\htecdcom}[1][z]{\ensuremath{\bar{F}^e_{h}(#1)}\xspace}        

\newcommand{\simplerootq}[1][]{\ensuremath{\rho^{(\forceofinterest)}_{#1}}\xspace}  
\newcommand{\simplerootperturbationq}[1][]{\ensuremath{\delta^{(\forceofinterest)}_{#1}}\xspace}  

\newcommand{\constantfactor}{\ensuremath{A(\forceofinterest)}\xspace}        



%% file: gerber.bib
@article{albrecher2010direct,
  title={A direct approach to the discounted penalty function},
  author={Albrecher, Hansj{\"o}rg and Gerber, Hans U and Yang, Hailiang},
  journal={North American Actuarial Journal},
  volume={14},
  number={4},
  pages={420--434},
  year={2010},
}

@BOOK{asmussen-RP,
  author       = {Asmussen, S{\o}ren and Albrecher, Hansj{\"o}rg},
  title        = {Ruin Probabilities},
  series       = {Advanced Series on Statistical Science \& Applied Probability, 14},
  edition      = {{S}econd},
  publisher    = {World Scientific},
  year         = {2010},
}

@ARTICLE{asmussen2004russian,
  author       = {Asmussen, S{\o}ren and Avram, Florin and Pistorius, Martijn R},
  title        = {Russian and {A}merican Put Options Under Exponential Phase-Type {L}\'evy Models},
  journal      = {Stochastic Processes and their Applications},
  volume       = {109},
  year         = {2004},
  number       = {1},
  pages        = {79--111},
}

@article{badescu2005surplus,
  title={The surplus prior to ruin and the deficit at ruin for a correlated risk process},
  author={Badescu, Andrei L and Breuer, Lothar and Drekic, Steve and Latouche, Guy and Stanford, David A},
  journal={Scandinavian Actuarial Journal},
  volume={2005},
  number={6},
  pages={433--445},
  year={2005},
}

@ARTICLE{biffis2010anote,
  author       = {Biffis, Enrico and Kyprianou, Andreas E.},
  title        = {A note on scale functions and the time value of ruin for {L}\'{e}vy insurance risk processes},
  journal      = {Insurance: Mathematics \& Economics},
  volume       = {46},
  year         = {2010},
  number       = {1},
  pages        = {85--91},
}

@ARTICLE{biffis2010generalization,
  author       = {Biffis, Enrico and Morales, M.},
  title        = {On a generalization of the Gerber-Shiu function to pathdependent penalties},
  journal      = {Insurance: Mathematics \& Economics},
  volume       = {46},
  year         = {2010},
  number       = {1},
  pages        = {92--97},
}

@ARTICLE{breuer2011generalised,
  author       = {Breuer, Lothar and Badescu, Andrei L.},
  title        = {A generalised Gerber--Shiu measure for Markov-additive risk processes with phase-type claims and capital injections},
  journal      = {Scandinavian Actuarial Journal},
 volume = {2014},
	number = {2},
	pages = {93--115},
	year = {2014},
}

@article{chiu2003time,
  title={The time of ruin, the surplus prior to ruin and the deficit at ruin
for the classical risk process perturbed by diffusion},
  author={Chiu, S. and Yin, C.},
  journal={Insurance: Mathematics and
Economics},
  volume={33},
  number={1},
  pages={59--66},
  year={2003},
}

@article{czarna2018fluctuation,
  title={Fluctuation identities for omega-killed Markov additive processes and dividend problem},
  author={Czarna, Irmina and Kaszubowski, Adam and Li, Shu and Palmowski, Zbigniew},
  journal={arXiv preprint arXiv:1806.08102},
  year={2018}
}

@BOOK{embrechts-MEE,
  author       = {Embrechts, Paul and Kl{\"u}ppelberg, Claudia and Mikosch, Thomas},
  title        = {Modelling Extremal Events: for Insurance and Finance},
  series       = {Applications of Mathematics},
  volume       = {33},
  publisher    = {Springer-Verlag},
  year         = {1997},
}

@article{feng2014potential,
  author={Feng, Runhuan and Shimizu, Yasutaka},
  title={Potential measures for spectrally negative Markov additive processes with applications in ruin theory},
  journal={Insurance: Mathematics \& Economics},
  volume={59},
  pages={11--26},
  year={2014},
}

@book{foss-IHTSD,
  title={An Introduction to Heavy-Tailed and Subexponential Distributions},
  author={Foss, Sergey and Korshunov, Dmitry and Zachary, Stan and others},
  volume={6},
  year={2011},
  publisher={Springer},
  edition = {2nd},
}

@ARTICLE{gerber1998on,
  author       = {Gerber, Hans U. and Shiu, Elias S. W.},
  title        = {On the time value of ruin},
  journal      = {North American Actuarial Journal},
  volume       = {2},
  year         = {1998},
  number       = {1},
  pages        = {48--78},
}

@ARTICLE{gerber2005time,
  author       = {Gerber, Hans U. and Shiu, Elias S. W.},
  title        = {The time value of ruin in a {S}parre {A}ndersen model},
  journal      = {North American Actuarial Journal},
  volume       = {9},
  year         = {2005},
  number       = {2},
  pages        = {49--54},
}

@ARTICLE{hua2013ruin,
  author       = {Hua, Dong and Zaiming, Liu},
  title        = {The ruin problem in a renewal risk model with two sided jumps},
  journal      = {Mathematical and Computer Modeling},
  volume       = {57},
  year         = {2013},
  number       = {3-4},
  pages        = {800--811},
}

@PhdThesis{ivanovs,
author = {Ivanovs, Jevgenijs},
title = {One-sided {M}arkov {A}dditive {P}rocesses and Related Exit Problems},
school = {Universiteit van Amsterdam},
year = {2011},
}

@ARTICLE{ivanovs2012occupation,
  author       = {Ivanovs, Jevgenijs and Palmowski, Zbigniew},
  title        = {Occupation densities in solving exit problems for {M}arkov additive processes and their reflections},
  journal      = {Stochastic Processes and their Applications},
  volume       = {122},
  year         = {2012},
  number       = {9},
  pages        = {3342--3360},
}

@article{ivanovs2014potential,
  title={Potential measures of one-sided {M}arkov additive processes with reflecting and terminating barriers},
  author={Ivanovs, Jevgenijs},
  journal={Journal of Applied Probability},
  volume={51},
  number={4},
  pages={1154--1170},
  year={2014},
}

@ARTICLE{jacobsen2005time,
  author       = {Jacobsen, M.},
  title        = {The time to ruin for a class of Markov additive risk process with two-sided jumps},
  journal      = {Advances in Applied Probability},
  volume       = {37},
  year         = {2005},
  number       = {4},
  pages        = {963--992},
}

@article{ji2010gerber,
  title={The Gerber--Shiu penalty functions for two classes of renewal risk processes},
  author={Ji, Lanpeng and Zhang, Chunsheng},
  journal={Journal of Computational and Applied Mathematics},
  volume={233},
  number={10},
  pages={2575--2589},
  year={2010},
}

@book{kato-PTLO,
  title={Perturbation Theory for Linear Operators},
  author={Kato, Tosio},
  year={1995},
  publisher={Berlin: Springer-Verlag},
  note={Corrected Printing of the 2nd Edition of 1980},
}

@misc{kolkovska2015gerber,
  author       = {Kolkovska, Ekaterina Todorova and Gonz\'{a}lez, Ehyter M. Mart\'{i}n},
  title        = {Gerber-Shiu functionals for two-sided jumps risk processes perturbed by an $\alpha$-stable motion},
  journal      = {Cimat Report},
  volume       = {No I-15-02/16-06-2015.},
  year         = {2015},
 }

@book{kyprianou-GSRT,
  title={Gerber--Shiu risk theory},
  author={Kyprianou, Andreas E},
  year={2013},
  publisher={Springer Science \& Business Media},
}

@article{labbe2009expected,
  title={The expected discounted penalty function under a risk model with stochastic income},
  author={Labb{\'e}, Chantal and Sendova, Kristina P},
  journal={Applied Mathematics and Computation},
  volume={215},
  number={5},
  pages={1852--1867},
  year={2009},
}

@ARTICLE{li2005gerber,
  author       = {Li, S. J. and Garrido, J.},
  title        = {The Gerber-Shiu function in a Sparre Andersen risk process perturbed by diffusion},
  journal      = {Scandinavian Actuarial Journal},
  volume       = {2005},
  year         = {2005},
  number       = {3},
  pages        = {161--186},
}

@article{lu2007expected,
  title = {The expected discounted penalty at ruin for a {M}arkov-modulated risk process perturbed by diffusion},
  author={Lu, Y. and Tsai, C. C.-L.},
  journal={North American Actuarial Journal},
  volume={11},
  number={2},
  pages={136--152},
  year={2007},
  publisher={Taylor \& Francis},
}

@ARTICLE{olvera2011transition,
  author       = {Olvera-Cravioto, Mariana and Blanchet, Jose and Glynn, Peter},
  title        = {On the transition from heavy traffic to heavy tails for the {$M/G/1$} queue: the regularly varying case},
  journal      = {The Annals of Applied Probability},
  volume       = {21},
  year         = {2011},
  number       = {2},
  pages        = {645--668},
}

@PhdThesis{vatamidou,
author = {Vatamidou, Eleni},
title = {Error analysis of stuctured {M}arkov chains},
school = {Eindhoven University of Tehcnology},
year = {2015},
}

@ARTICLE{vatamidou2013correctedrisk,
  author       = {Vatamidou, Eleni and Adan, Ivo Jean Baptiste Fran{\c{c}}ois and Vlasiou, Maria and Zwart, Bert},
  title        = {Corrected phase-type approximations of heavy-tailed risk models using perturbation analysis},
  journal      = {Insurance: Mathematics \& Economics},
  year         = {2013},
  volume       = {53},
  number       = {2},
  pages        = {366--378},
}

@article{xing2008time,
  title={On the time to ruin and the deficit at ruin in a risk model with double-sided jumps},
  author={Xing, Xiaoyu and Zhang, Wei and Jiang, Yiming},
  journal={Statistics \& Probability Letters},
  volume={78},
  number={16},
  pages={2692--2699},
  year={2008},
  publisher={Elsevier}
}

@article{zhang2010perturbed,
  title={The perturbed compound Poisson risk model with two-sided jumps},
  author={Zhang, Zhimin and Yang, Hu and Li, Shuanming},
  journal={Journal of Computational and Applied Mathematics},
  volume={233},
  number={8},
  pages={1773--1784},
  year={2010},
  publisher={Elsevier}
}
